\documentclass[12pt,a4paper]{amsart}
\usepackage[margin=1.2in]{geometry}
\usepackage{amsmath}
\usepackage{amsthm}
\usepackage{amsfonts}
\usepackage{amssymb}
\usepackage{amscd}
\numberwithin{equation}{section}
\newtheorem{theorem}{Theorem}[section]

\newtheorem{lemma}[theorem]{Lemma}
\newtheorem{prop}[theorem]{Proposition}
\theoremstyle{definition}
\newtheorem{defi}{Definition}[section]
\newtheorem{rem}[defi]{Remark}

\newcommand\half{\frac{1}{2}}

\newcommand\ov{\overline}

\newcommand\g{\mathfrak g}

\newcommand\h{\mathfrak h}

\newcommand\n{\mathfrak n}
\newcommand\bb{\mathfrak b}
\newcommand\D{\Delta}
\renewcommand\l{\lambda}
\newcommand\Dp{\Delta^+}

\renewcommand\d{\delta}

\renewcommand\a{\alpha}
\renewcommand\aa{\mathfrak a}
\renewcommand\b{\bb}

\renewcommand\i{{\mathfrak i}}

\newcommand\ganz{\mathbb Z}

\newcommand\s{\sigma}
\renewcommand\L{\Lambda}

\renewcommand\aa{\mathfrak a}
\newcommand\la{\lambda}

\newcommand\C{\mathbb C}
\newcommand\R{\mathbb R}
\newcommand\si{\sigma}

\newcommand\va{|0\rangle}

\newcommand{\sdim}{\text{\rm sdim}}

\begin{document}
\title[Very Strange Formula for  Lie superalgebras]{Dirac operators and the Very Strange Formula for  Lie superalgebras}
\author[Kac, M\"oseneder and Papi]{Victor~G. Kac\\Pierluigi M\"oseneder Frajria\\Paolo  Papi}

\begin{abstract}
Using a super-affine version of Kostant's cubic Dirac operator, we prove a very strange formula for quadratic  finite-dimensional Lie superalgebras with a reductive even subalgebra.
\end{abstract}
\maketitle
\section{Introduction}
The goal of this paper is to provide an approach to  the strange and very strange formulas  for a wide class of finite dimensional Lie superalgebras. Let us recall what these formulas are in the even case.
Let $\g$ be a finite-dimensional complex simple Lie algebra. Fix a Cartan subalgebra $\h\subset\g$ and let $\Dp$ be a  set of positive roots for the set $\D$ of $\h$-roots in $\g$. Let $\rho=\tfrac{1}{2}\sum_{\a\in\Dp}\a$ be the corresponding Weyl vector. Freudenthal and de Vries discovered  in \cite{F} the following remarkable relation between the square length of $\rho$ in the Killing form $\kappa$ and the dimension of $\g$:
\begin{equation*}\label{fsf}\kappa(\rho,\rho)=\frac{\dim\g}{24}.\end{equation*}
They called this the {\it strange formula}. It can be proved in several very different ways (see e.g. \cite{FS}, \cite{B}), and it plays an important role in  the proof of   the Macdonald identities. Indeed, the very strange formula enters as a transition factor between the Euler product $\varphi(x)=\prod\limits_{i=1}^\infty(1-x^i)$ and 
Dedekind's $\eta$-function $\eta(x)=x^{\frac{1}{24}}\varphi(x)$.\par In \cite{Kaceta1} Kac gave  a representation theoretic interpretation of the  Macdonald identities as  denominator identities for an affine Lie algebra. Moreover, 
using the modular properties of characters of the latter algebras,  æ he provided in \cite{Kaceta} a multivariable generalization of them and a corresponding transition identity that he named the
{\it very strange formula}. Here the representation theoretic interpretation of the formulas involves an  affine Lie algebra, which is built up from  a simple Lie algebra endowed with a  finite order automorphism. To get the  very strange formula from a ``master formula'' it was also required that  the  characteristic polynomial of the automorphism has rational coefficients. A more general 
form, with no  rationality hypothesis, is proved in \cite{KP}, where it is also used to estimate  the asymptotic behavior at cusps of the modular forms involved in the character of an highest weight module. 
Let us state this version of the very strange formula, for simplicity of exposition just in the case of  inner automorphisms.
Let  $\sigma$ be an automorphism of order $m$ of type 
$(s_0,s_1,\ldots,s_n;1)$ (see \cite[Chapter 8]{Kac}). Let 
$\g=\oplus_{\bar j\in\ganz/m\ganz}\,\g^{\ov j}$ be the eigenspace 
decomposition with respect to $\s$. Define  $\l_s\in\h^*$
by $\kappa(\l_s,\a_i)=\frac{s_i}{2m},\,1\leq i\leq n$; here $\{\a_1,\ldots,\a_n\}$ is the set of simple roots of $\g$. Then 
\begin{equation}\label{vsfk} \kappa(\rho-\l_s,\rho-\l_s)=
\frac{\dim\g}{24}-\frac{1}{4m^2}\sum_{j=1}^{m-1}j(m-j)\dim \g^{\ov j}.
\end{equation}

Much more recently, we provided a vertex algebra approach to this formula (in a slight generalized version where an  elliptic automorphism is considered, cf. \cite{KMP}) æas a byproduct of our attempt to reproduce Kostant's theory of the cubic Dirac operator in  affine 
setting. Our proof is based on two main ingredients:
\begin{enumerate} 
\item[(a).] An explicit vertex algebra isomorphism $V^k(\g)\otimes F(\overline \g)\cong V^{k+g,1}(\g)$, where $V^k(\g)$ is the affine vertex algebra of noncritical level $k$ and  $F(\overline \g)$ is the fermionic 
vertex algebra of $\g$ viewed as a purely odd space.
\item[(b).] A nice formula for the $\l$-bracket of the Kac-Todorov Dirac field $G\in V^{k+g,1}(\g)$ with itself.
\end{enumerate}
\par
Indeed, using (a), we can let the zero mode $G_0$ of $G$ act on the tensor product of representations of $V^k(\g), F(\ov\g)$. Since we are able to compute $G_0\cdot(v\otimes 1)$, where $v$ is an  highest
weight vector of an highest weight  module for the affinization of $\g$, the expression for $[G_\l G]$ obtained in step  (b) yields a formula which can be  recast  in the form
\eqref{vsfk}æ (cf. \cite[Section 6]{KMP}).
\par
Now we discuss our work in the super case. A finite dimensional  Lie superalgebra $\g=\g_0\oplus\g_1$ is called quadratic if it carries a supersymmetric bilinear form (i.e. symmetric on $\g_0$, 
skewsymmetric on $\g_1$, and $\g_0$ is orthogonal to $\g_1$), which is 
non-degenerate and invariant. We say that a  complex quadratic  Lie superalgebra $\g$ is of {\it basic type} if  $\g_0$ is a reductive subalgebra of $\g$. 
In Theorem \ref{sf}   we  prove a {\it very strange formula} (cf. \eqref{vstrangeformula}) for basic type Lie superalgebras endowed with an indecomposable elliptic   automorphism (see Definition \ref{de})
which preserves the invariant form. When the automorphism is the identity, this formula specializes to the strange formula \eqref{strangeformula}, which has been proved for Lie superalgebras of defect  zero in \cite{Kaceta}æ and for general basic classical Lie superalgebras in 
\cite{KW}, using case by case combinatorial calculations. The proof using the Weyl character formula as in \cite{F}æ or the proof using modular forms as in \cite{KP} are not applicable in this setting.

Although the proof proceeds along the lines of what we did in  \cite{KMP} for Lie algebras, we have to face  several technical difficulties. We single out two of them. First, we have to 
build up a twisted Clifford-Weil  module for $F(\ov\g)$; this requires a careful  choice of a maximal isotropic subspace in $\g^{\ov 0}$.
In Section 2  we prove that the class of  Lie superalgebras of basic type  is closed under taking fixed points of automorphisms and in Section 3 we show that  Lie superalgebras of basic type  admit a triangular decomposition. This implies the existence a ``good'' maximal isotropic subspace.\par
Secondly, the isomorphism in (a) is given by  formulas which are different w.r.t. the even case, and this makes  subtler the computation of the  square of the Dirac field under $\l$-product. We  have also obtained several simplifications with respect to the exposition given in \cite{KMP}.
\par
Some of our results have been (very  sketchily) announced in \cite{P}.

\section{Setup}\label{2} Let $\g=\g_0\oplus\g_1$ be a finite dimensional  Lie superalgebra of basic type, i.e., 
\begin{enumerate}\item
$\g_0$ is reductive subalgebra of $\g$, i.e., the adjoint representation of $\g_0$ on $\g$ is completely reducible;
\item
$\g$ is {\it quadratic}, i.e., $\g$ admits a  nondegenerate invariant supersymmetric bilinear form  $(\cdot,\cdot)$. 
\end{enumerate}
Note that  condition (1) implies
that $\g_0$ is a reductive Lie algebra and that 
$\g_1$ is completely reducible as a $\g_0$-module. 
Examples are given by the  simple basic classical Lie superalgebras and the contragredient finite dimensional  Lie superalgebras with a symmetrizable Cartan matrix (in particular, $gl(m,n)$). There are of course examples of different kind, like a symplectic vector space regarded as a purely odd abelian Lie superalgebra. An inductive classification is provided 
in \cite{Be}. 
 \par
We say that $\g$ is $(\cdot,\cdot)$-irreducible if the form restricted to a proper ideal is dege\-nerate.
\begin{defi}\label{de}
An automorphism $\s$ of $\g$ is said indecomposable if $\g$ cannot be decomposed as an orthogonal  direct sum of two nonzero $\s$-stable ideals. \par
We say that $\s$ is elliptic if it  is diagonalizable with modulus $1$ eigenvalues. 
\end{defi}

 Let $\sigma$ be an indecomposable elliptic automorphism of $\g$ which is parity preserving and leaves the form invariant. 
  If $j\in\R$, set $\bar j=j +\ganz\in\R/\ganz$. Set $\g^{\bar j}=\{x\in\g\mid \sigma(x)=e^{2\pi\sqrt{-1}j}x\}$. Let $\h^0$ be a  Cartan subalgebra of $\g^{\ov0}$. 

  \begin{prop}\label{isreductive} If $\g$ is of basic type, then 
  $\g^{\ov 0}$ is of basic type.
    \end{prop}
  \begin{proof}Since $\s$ is parity preserving, it induces an automorphism of $\g_0$. Since $\g_0$ is reductive, we have that   $\g^{\ov 0}_0$ is also reductive. Since $\s$ preserves the invariant bilinear  form, we have that  $(\g^{\ov i},\g^{\ov j})\ne 0$ if and only if $\ov i =-\ov j$. Thus $(\cdot,\cdot)_{|\g^{\ov 0}\times \g^{\ov 0}}$ is nondegenerate. Since $\g^{\ov 0}_0$ is reductive, $\h^0$ is abelian, thus it is contained in a Cartan subalgebra $\h$ of $\g$. Since $\g_1$ is completely reducible as a $\g_0$-module, $\h$ acts semisimply on $\g_1$, hence $\h^0$ acts semisimply on $\g^{\ov 0}_1$. Thus $\g^{\ov 0}_1$ is a semisimple $\g^{\ov 0}_0$-module.

  \end{proof}
  
It is well-known that we can choose as Cartan subalgebra for $\g$ the centralizer $\h$ of $\h^0$ in $\g_0$. 
In particular we have that $\s(\h)=\h$. If $\aa$ is any Lie superalgebra, we let $\mathfrak z(\aa)$ denote its center.

\section{The structure of basic type Lie superalgebras}

The goal of this section is to prove that a basic type Lie superalgebra admits a triangular decomposition. We will apply this result in the next sections to $\g^{\ov 0}$, 
which, by  Proposition \ref{isreductive}, is of basic type.\par
Since $\g_0$ is reductive, we can fix a Cartan subalgebra $\h\subset\g_0$ and a set of positive roots for $\g_0$.  If $\l\in\h^*$, let $h_\l$ be, as usual, the unique element of $\h$ such that $(h_\l,h)=\l(h)$ for all $h\in\h$.
Let $V(\l)$ denote the irreducible representation of $\g_0$ with highest weight $\l\in(\h)^*$.

Then  we can write 
\begin{equation}\label{1}
\g=\h+[\g_0,\g_0]\oplus\sum_{ \l\in\h^*} V(\l).
\end{equation}
 Decompose now the Cartan subalgebra $\h$ of $\g$ as
$$\h=\h'\oplus\h'',\qquad \h'=\h\cap[\g,\g].$$

Let $M_{triv}$ be the isotypic component of the trivial $[\g_0,\g_0]$-module in $\g_1$. 
Decompose  it into isotypic components for $\g_0$ as 
\begin{equation}\label{isotyp}
M_{triv}=\oplus_{\l\in\Lambda} M(\l)=M(0)\oplus M_{triv}',
\end{equation}
 where 
$M_{triv}'=\oplus_{0\ne \l\in\Lambda} M(\l)$. 
Then 
\begin{align}
\g^{(1)}&:=[\g,\g]\notag\\&=\h' + [\g_0,\g_0]\oplus\sum_{ 0\ne\l\in\h^*}V(\l)\notag\\
&=\h' + [\g_0,\g_0]\oplus M_{triv}'\oplus\sum_{\l\in\h^*,\dim V(\l)>1}V(\l)\label{guno}.
\end{align}

\begin{lemma}\label{unid} If in decomposition \eqref{1} we have that  $\dim V(\l)=1$, then $\l(\h')=0$.
\end{lemma}
\begin{proof} Let $h\in\h'$. If $h\in [\g_0,\g_0]$ the claim is obvious; if $h\in [\g_1,\g_1]$ then we may assume that $h=[x_\mu,x_{-\mu}]=h_\mu$, $\mu$ being a 
$\h$-weight of $\g_1$. Assume first $\mu\pm\l\ne0$.  Then, for $v_{\lambda}\in V(\lambda)$, we have 
\begin{equation}\label{1dim}0=[v_{-\l},[v_\l,x_\mu]]=\mu(h_\l)x_\mu-[v_\l,[v_{-\l},x_\mu]]=\mu(h_\l)x_\mu
\end{equation}
so that $\mu(h_\l)=0$ or $\l(h_\mu)=0$. 
It remains to deal with the case $\mu=\pm\l$. We have
$$[[v_\l,v_{-\l}],v_{\pm \l}]= \pm ||\l||^2 v_{\pm\l}.$$
This finishes the proof, since $||\l||^2=0$. Indeed, $[v_\l,v_\l]$ is a weight vector of weight $2\l$. This  implies either $\l=0$, and we are done,  or $[v_\l,v_\l]=0$. In the latter case
$||\l||^2 v_\l=[[v_\l,v_{-\l}],v_\l]=0$ by the Jacobi identity.
\end{proof}
In turn, by Lemma \ref{unid},
\begin{align*}
\g^{(2)}&:=[\g^{(1)},\g^{(1)}]\\&=\h' + [\g_0,\g_0]\oplus(\sum_{\l\in\h^*,\,\dim V(\l)>1}V(\l)).
\end{align*}
Finally define
$$\underline\g=\g^{(2)}/ (\mathfrak z(\g) \cap \g^{(2)}).$$
\begin{lemma}\label{lf}
\
\begin{enumerate} 
\item The radical of the restriction of the invariant form to $\g^{(2)}$ equals $\mathfrak z(\g) \cap \g^{(2)}$.
\item $\underline \g$ is an orthogonal direct  sum of quadratic simple Lie superalgebras.
\end{enumerate}
\end{lemma}
\begin{proof} It suffices to show that if $x\in \g^{(2)}$ belongs to the radical of the (restricted) form, then it belongs to the center of $\g$.
We know that $(x,[y,z])=0$ for all $y,z\in \g^{(1)}$; invariance of the form implies that  $[x,y]$ belongs to the radical of the form restricted to $\g^{(1)}$. This in turn 
means that $([x,y],[w,t])=0$ for all $w,t\in \g$. Therefore, $0=([x,y],[w,t])=([[x,y],w],t)\,\forall\,t\in \g$. Since the form on $\g$ is nondegenerate, we have that 
$[x,y]\in\mathfrak z(\g)$ for any $y\in \g^{(1)}$. If $x\in\g_1$ and $y\in \g^{(1)}_0$, then  $[x,y]\in \mathfrak z(\g)\cap(\sum_{\l\in\h^*,\,\l_{|\h'}\ne 0}V(\l))=\{0\}$. This imples that $x$ commutes with $\g^{(1)}_0$. Since $x\in \sum_{\l\in\h^*,\,\l_{|\h'}\ne 0}V(\l)$, we have $x=0$. If $x\in\g_0$, then, if $y\in \sum_{\l\in\h^*,\,\l_{|\h'}\ne 0}V(\l)$, we have $[x,y]\in \mathfrak z(\g)\cap(\sum_{\l\in\h^*,\,\l_{|\h'}\ne 0}V(\l))=\{0\}$. If $y\in \g^{(2)}_0$, then $[x,y]\in[\g_0,\g_0]\cap\mathfrak z(\g)=\{0\}$. So $x\in\mathfrak z(\g^{(2)})$. This implies that $x\in\h'$, so it commutes also with $M_{triv}$ and $\h$, hence $x\in\mathfrak z(\g)$, as required.\par
To prove the second statement, it suffices to show that there does not exist an isotropic ideal in $\underline \g$. Indeed, if this is the case and $\i$ is a minimal ideal, then by minimality 
either $\i\subseteq \i^\perp$ or $\i\cap\i^\perp=\{0\}$. Since we have excluded the former case, we have $\underline \g=\i\oplus \i^\perp$ with  $\i$ a  simple Lie superalgebra endowed with a non degenerate form and we can conclude by induction. \par
Suppose that $\i$ is an isotropic ideal. If $x\in \g^{(2)}$, we let $\pi(x)$ be its image in  $\underline\g$. If $\i_1\ne\{0\}$, we have that there is $\pi({V(\l)})\subset \i_1$. We can choose an highest weight vector $v_{\l}$  in $V(\l)$ and a vector $v_{-\l}\in\g_1$ of weight $-\l$ such that $(v_\l,v_{-\l})=1$. Then $\pi( h_\l)=[\pi( v_\l),\pi( v_{-\l})]\in \i$. Note that $h_\l\not\in \mathfrak z(\g)$. In fact, since $\dim V(\l)>1$, $\l_{|\h\cap [\g_0,\g_0]}\ne0$. On the other hand, if $[h_\l,\g^{(2)})]\ne 0$, then there is $0\ne v_\mu\in( \g^{(2)})_\mu$ such that $[h_\l,v_\mu]=\l(h_\mu)v_\mu\ne0$. In particular $\pi( v_{\mu})\in \i$. Choose $v_{-\mu}\in (\g^{(2)})_{-\mu}$ such $(v_\mu,v_{-\mu})\ne 0$. Then $\pi( v_{-\mu})=-\frac{1}{\l(h_\mu)}[\pi( h_\l), \pi( v_{-\mu})]\in\i$. But then $\i$ is not isotropic. It follows that $h_\l\in\mathfrak z(\g^{(2)})=\mathfrak z(\g)\cap \g^{(2)}$, which is absurd. 
Then $\i=\i_0$, hence $\i\subset \mathfrak z(\underline \g_0)$. Since $[\i,\underline \g_1]=0$, we have that $\pi^{-1}(\i)\subset \mathfrak z(\g^{(2)})$. Since $\mathfrak z(\g^{(2)})=\mathfrak z(\g)\cap\g^{(2)}$, we have $\i=\{0\}$.
\end{proof}
\vskip10pt
At this point we have the following decomposition:
$$\g=(\h+[\g_0,\g_0])\,\oplus\sum_{\l\in\h^*,\,\dim V(\l)>1}V(\l)\oplus M_{triv}
. $$
Let, as in the proof of Lemma \ref{lf}, $\pi:\g^{(2)}\to \underline \g$ be the projection. Since $\mathfrak z(\g)\cap[\g_0,\g_0]=0$, we see that $[\g_0,\g_0]=\pi([\g_0,\g_0])$, hence we can see the set of positive roots for $\g_0$ as a set of positive roots for $\underline\g_0$. By Lemma \ref{lf}, we have  
\begin{equation}\label{decogbar}
\underline \g=\bigoplus\limits_{i=1}^k \underline \g(i),
\end{equation} with $\underline \g(i)$ simple ideals. It is clear that $\g$ acts on $\underline \g$ and that the projection $\pi$ intertwines the action of $\g_0$ on $\g^{(2)}_1$ with that on $\underline\g_1$. Since $[\g_0,\g_0]=\pi([\g_0,\g_0])$, we see that $\underline\g(i)_1$ is a $[\g_0,\g_0]$-module. Since the decomposition \eqref{decogbar} is orthogonal, we see that the $[\g_0,\g_0]$-modules $\underline\g(i)_1$ are inequivalent. It follows that  that $\mathfrak z(\g_0)$ stabilizes $\underline\g(i)_1$, thus $\underline\g(i)_1$ is a $\g_0$-module. 

 We now discuss the $\g_0$-module structure of $\underline\g(i)_1$.
 By the classification of simple Lie superalgebras, either $\underline \g(i)_1= V(i)$ with   $V(i)$ self dual irreducible $\underline \g_0$-module or there is a polarization (with respect to $(\cdot,\cdot)$)
$\underline \g(i)_1= V\oplus V^*$ with $V$ an irreducible $\underline \g_0$-module. 
 In the first case $\underline\g(i)_1$ is an irreducible $\g_0$-module. In the second case, since the action of $\h$ is semisimple, $\underline\g(i)$ decomposes as $V_1(i)\oplus V_2(i)$, with $V_j(i)$ ($j=1,2$) irreducible $\g_0$-modules. If $V_1(i)$ is not self dual, then the decomposition $\underline\g(i)=V_1(i)\oplus V_2(i)$ is a polarization. If $V_1(i)=V_1(i)^*$ and $V_2(i)=V_2(i)^*$ then  the center of $\g_0$ acts trivially on $\underline\g(i)_1$, thus $V$ and $V^*$ are actually $\g_0$-modules. We can therefore choose $V_1(i)=V$ and $V_2(i)=V^*$.

The simple ideals $\underline \g(i)$ are basic classical Lie superalgebras. By the classification of such algebras (see \cite{Kacsuper}) there is a contragredient  Lie superalgebra $\tilde \g(i)$ such that $\underline \g(i)=[\tilde \g(i),\tilde \g(i)]/\mathfrak z([\tilde \g(i),\tilde \g(i)])$. 
 Choose  Chevalley generators $\{\tilde e_j,\tilde f_j\}_{j\in J(i)}$ for $\tilde \g(i)$. Let $\underline e_j, \underline f_j$ be their image in $\underline \g(i)$. 
 
 We claim that $\underline e_j, \underline f_j$ are $\h$-weight vectors.
The vectors  $\underline e_j, \underline f_j$ are root vectors for $\underline \g(i)$.  If  the roots of $\underline e_j, \underline f_j$  have multiplicity one, then, $\underline e_j, \underline f_j$ must be $\h$-stable, hence they are $\h$-weight vectors. If there are roots of higher multiplicity then $\underline \g(i)$ is of type $A(1,1)$. Let $\ov d$ be the derivation on $\underline \g(i)$ defined by setting $\ov d(\underline \g_0)=0$, $\ov d(v)=v$ for $v\in V_1(i)$, and $\ov d(v)=-v$ for $v\in V_2(i)$. Then it is not hard to check that $\tilde\g(i)=\C d\oplus \C c\oplus \underline \g(i)$ with bracket defined as in Exercise 2.10 of \cite{Kac}. If the roots of $\underline e_j,\underline f_j$ have multiplicity two, then $\tilde e_j,\tilde f_j$ are odd root vectors of $\tilde \g(i)$. In particular $\tilde e_j$ is in $V_1(i)$  and $\tilde f_j$ is in $V_2(i)$. This implies that $\underline e_j$ is in $V_1(i)$ and $\underline f_j$ is in $V_2(i)$. Since $\mathfrak z(\g_0)$ acts as multiple of the identity on $V_1(i)$ and $V_2(i)$, we see that $\underline e_j,\underline f_j$ are $\h$-weight vectors also in this case.

 Since $\pi$ restricted to $[\g_0,\g_0]+\g^{(2)}_1$ is an isomorphism, we can define $e_j,f_j$ to be the unique elements of $[\g_0,\g_0]+\g^{(2)}_1$ such that $\pi(e_j)=\underline e_j$ and $\pi(f_j)=\underline f_j$.
 Since $\underline e_j,\underline f_j$ are $\h$-weight vectors, we have that $e_j,f_j$ are root vectors for $\g$.

 Set $J=\cup_i J(i)$. We can always assume that
the positive root vectors of $\underline\g_0$ are in the algebra spanned by $\{\underline e_j\mid j\in J\}$.

 Let $\a_j\in\h^*$ be the weight of $e_j$.  We note that the weight of $f_j$ is $-\a_j$ for any $j\in J$. One way to check this is the following: if $j\in J(i)$, there is an invariant form $<\cdot,\cdot>$ on $\tilde \g(i)$ such that $<\tilde e_j,\tilde f_j>\ne 0$. Since $\underline \g(i)$ is simple, the form $(\cdot,\cdot)$ is a (nonzero) multiple of the form induced by $<\cdot,\cdot>$. In particular $(\underline e_j,\underline f_j)\ne 0$. Since $(e_j,f_j)=(\underline e_j,\underline f_j)$, we see that the root of $f_j$ is $-\a_j$.

 Subdivide   $\Lambda$ in \eqref{isotyp} as $\Lambda=\{0\}\cup\Lambda^+\cup\Lambda^-$ with $\L^+\cap\L^-=\emptyset$ and $\Lambda^-= -\Lambda^+$ (which is possible since  the form $(\cdot,\cdot)$ is nondegenerate on $M'_{triv}$). 
Choose a basis  $\{e^\l_i\mid i=1,\ldots, \dim M(\l)\}$ in $M(\l)$ for $\l\in \L^+$ and let $\{f^\l_i\}\subset M(-\l)$ be the dual basis.  Also $(\cdot,\cdot)_{|M(0)\times M(0)}$ is nondegenerate, hence we can find a polarization $M(0)=M^+\oplus M^-$. Let $\{e_i^0\}$ and $\{f_i^0\}$ be a basis of $M^+$ and its dual basis in $M^-$, respectively. 

We now check that relations
\begin{equation*} \label{r1}
[e_i,f_j]=\d_{ij}h_i,\ i,j\in J\quad
[e_i,f^\l_j]=[e^\l_i,f_j]=0,\ j\in J\quad
[e^\l_i,f^\mu_j]=\d_{\l,\mu}\d_{i,j}h^\l_i,
\end{equation*}
 hold for $\{e_j,f_j\}_{j\in J}\cup \{e^\l_i,f^\l_i\}_{\l\in\L^+\cup\{0\}}$.  Assume now $i\ne j$, $i,j\in J$. Then, since $[\underline e_{i},\underline f_{j}]=0$,  $[e_i,f_j]\in \mathfrak z(\g)\cap\g^{(2)}\subset \h'$.  This implies that $\a_i=\a_j$ so  $[e_i,f_j]\in\C h_{\a_i}$. If $e_i$ is even then $\a_i$ is a root of $\g_0$ so $\pi(h_{\a_i})\ne 0$, hence $[e_i,f_j]=0$. If $e_i$ is odd and $\pi(h_{\a_i})=0$, since $[\underline e_i,\underline f_i]=(\underline e_i,\underline f_i)\pi(h_{\a_i})=0$, we have that $[\underline e_i,\underline f_j]=0$ for any $j\in J$. In particular, $\underline e_i$ is a lowest weight vector for $\underline \g_0$. On the other hand, since $h_{\a_i}\in\mathfrak z(\g)$, we have in particular that $\a(h_{\a_i})=0$ for any root of $\g_0$. This implies that $\C\underline e_i$ is stable under the adjoint action of $\underline \g_0$. This is absurd since $\underline \g_1$ does not have one-dimensional  $\underline \g_0$-submodules.
 
   If $\l,\mu\in \L^+\cup\{0\}$, then $[e_i^\l,f_j^\mu]$ is in the center of $\g_0$, hence $[e^\l_i,f^\mu_j]=\d_{i,j}\d_{\l,\mu}h_\l$. Moreover it is obvious that $[e^\l_h,f_j]=[e_j,f^\l_h]=0$ if $e_j,f_j$ are even. It remains to check that  $[e^\l_h,f_{j}]=[e_{j},f^\l_h]=0$ when $e_j,f_j$ are odd.

This follows from the more general 
\begin{lemma}\label{Mcommg2}
$$
[M_{triv},\g^{(2)}_1]=0.
$$
\end{lemma}
\begin{proof}Choose $x\in M(\l)$ and $y\in V(\mu)$ with $\dim V(\mu)>1$. It is enough to show that
$([x,y],z)=0$ for any $z\in \g_0$. Observe that, since $\C x$ and $V(\mu)^*$ are inequivalent as $\g_0$-modules, we have that $(x,V(\mu))=0$. Since $([x,y],z)=(x,[y,z])$ and $[y,z]$ is in $V(\mu)$, we have the claim.
\end{proof}

\vskip 10 pt
The outcome of the above construction is that we have a triangular decomposition
\begin{equation}\label{dd}
\g= \n+ \h+ \n_-,
\end{equation}
where $\n$ (resp. $\n_-$) is the algebra generated by  $\{e_j,\mid j\in J\}\cup\{e^\l_i\mid \l\in\L^+\cup\{0\}\}$ (resp. 
$\{f_j,\mid j\in J\}\cup\{f^\l_i\mid \l\in\L^+\cup\{0\}\}$).
By Lemma \ref{Mcommg2}, we see that $[e^\l_h, e_{j}]=0$. It follows that 
$$
\n=\n_{triv}\oplus \mathfrak e,
$$
where $\n_{triv}$ is the algebra generated by  $\{e^\l_i\mid \l\in\L^+\cup\{0\}\}$ and $\mathfrak e$ is the algebra generated by $\{e_j,\mid j\in J\}$. Notice that 
$$\n_{triv}=M^+\oplus \sum_{\l\in\L^+}M(\l).
$$
This follows from the fact that the right hand side is an abelian subalgebra. Since $\mathfrak e\subset \g^{(2)}$, we have the orthogonal decomposition
\begin{equation}\label{decompn0}
\n=M^+\oplus (\sum_{\l\in\L^+}M(\l))\oplus (\n\cap \g^{(2)}). 
\end{equation}

 Choose any maximal isotropic subspace $\h^+$ in $\h$. The previous constructions imply the following fact.

\begin{lemma} \label{isotropic}$\h^++ \n$ is a maximal isotropic subspace in $\g$. \end{lemma}
\begin{proof}
We first prove that $\n$ is isotropic. By \eqref{decompn0}, it is enough to check that $M^+\oplus(\sum_{\l\in\L^+}M(\l))$ and $(\n\cap \g^{(2)})$ are isotropic. By construction  $M^+$ is isotropic. Moreover, if $\l\ne -\mu$, then $M(\l)^*$ and $M(\mu)$ are inequivalent, thus  $(M(\l),M(\mu))=0$. This implies that $M(\l)$ is isotropic if $\l\ne0$ and $(M(\l), M(\mu))=0$ if $\l\ne\mu$, $\l,\mu\in\L^+\cup\{0\}$.

 If $x,y\in\n\cap\g^{(2)}$ and $\pi(x)\in\underline \g(i)$, $\pi(y)\in\underline \g(j)$ with $i\ne j$, then $(x,y)=(\pi(x),\pi(y))$ $=0$. If $i=j$, let  $p:[\tilde \g(i),\tilde \g(i)]\to \underline \g(i)$ be the projection.  Let $\tilde \n(i)$ be the algebra spanned by the $\{\tilde e_j\}_{j\in J(i)}$. Then $\pi(x),\pi(y)\in p(\tilde \n(i))$.  Recall that the weights of $\tilde \n(i)$ are a set of positive roots for $\tilde \g(i)$, and, since $\tilde \g(i)$ is contragredient $\a$ and $-\a$ cannot be both positive roots for $\tilde \g(i)$. This implies that $\tilde \n(i)$ is an isotropic subspace of $\tilde \g(i)$ for any invariant form of $\tilde \g(i)$. Since $\underline \g(i)$ is simple, $(\cdot,\cdot)$ is induced by an invariant form on $\tilde \g(i)$ so $p(\tilde n(i))$ is isotropic.
 
  Clearly, $(\h,\n)=(\h,\n\cap\g_0)=0$, since $\n\cap\g_0$ is the nilradical of a Borel subalgebra. Note that $(\n,\g)=(\n,\n_-)$ so $\n$ and $\n_-$ are non degenerately paired. Thus $\n$ is a maximal isotropic subspace of $\n+\n_-$. Since $\h$ and $\n+\n_-$ are orthogonal, the result follows.
\end{proof}
\begin{prop}\label{triangular}
\begin{equation}\label{triang}
\g= \n\oplus \h\oplus \n_-.
\end{equation}
\end{prop}
\begin{proof} Having \eqref{dd} at hand, it remains to prove that the sum is direct. This follows  from Lemma \ref{isotropic}:
indeed, if $x\in \n\cap\n_-$, then $x$ would be in the radical of the form.\end{proof}



\section{The super affine vertex algebra}

Set $\bar\g=P\g$, where $P$ is the parity reversing functor.   In the following, we refer the reader to \cite{KacD} for basic definitions and notation regarding Lie conformal and vertex algebras. In particular, for the reader's convenience we recall Wick's formula, which will be used several times in the following. Let $V$ be a vertex algebra, then 
\begin{equation}\label{Wick}
[a_\l:bc:]=\, :[a_\l b]c:+p(a,b):b[a_\l 
c]:+\int_0^\l[[a_\l b]_\mu c]d\mu.
\end{equation}

Consider the
conformal algebra
$R=(\C[T]\otimes\g)\oplus(\C[T]\otimes\bar \g)\oplus\C K\oplus\C \bar K$  with 
$\lambda$-products
\begin{align}\label{prod}
&[a_\lambda b]=[a,b]+\lambda (a,b)K,\\
&[a_\lambda \bar b]=\overline{[a,b]},\quad
[\bar a_\lambda b]=p(b)\overline{[a,b]},\label{seconda}\\\label{terza}
&[\bar a_\lambda \bar b]=(b,a)\bar K,
\end{align}
$K,\,\bar K$ being even central 
elements.
 Let $V(R)$ be the corresponding universal vertex algebra, and denote by $V^{k,1}(\g)$ its quotient by the ideal generated by $K-k|0\rangle$ and $\overline K - |0\rangle$. The vertex algebra $V^{k,1}(\g)$ is called  the super affine vertex algebra of level $k$. The relations are the same used in \cite{KMP}æ for even variables. We remark that the order of $a,b$ in the r.h.s. of  \eqref{terza} is relevant.
 
Recall that one defines the  current Lie  
conformal superalgebra $Cur(\g)$ as 
$$ Cur(\g)=(\C[T]\otimes \g)+\C\mathcal K
$$ with $T(\mathcal K)=0$ and the $\l$-bracket defined for $a,b\in1\otimes 
\g$  by
$$ [a_\l b]=[a,b]+\l(a,b)\mathcal K,\quad [a_\l \mathcal K]=[\mathcal K_\l \mathcal K]=0.
$$ Let $V(\g)$ be its universal enveloping vertex algebra. The quotient $V^k(\g)$ of $V(\g)$ by the ideal generated by $\mathcal K-k|0\rangle$ is called the level $k$ affine vertex algebra.

If $A$ is a superspace equipped with a skewsupersymmetric bilinear form $<\cdot,\cdot>$ one also has the Clifford Lie conformal superalgebra $$C(A)=(\C[T]\otimes A)+\C \ov{\mathcal K}$$
 with $T( \ov{\mathcal K})=0$ and the $\l$-bracket defined for $a,b\in1\otimes A$  by
$$ [a_\l  b]=<a,b>\ov{\mathcal K},\quad [ a_\l  \ov{\mathcal K}]=[ \ov{\mathcal K}_\l \ov{\mathcal K}]=0.
$$

Let $V$ be the universal enveloping vertex algebra of $C(A)$. 
The quotient  of $V$ by the ideal generated by $ \ov{\mathcal K}-|0\rangle$ is denoted by $F(A)$. Applying this construction to $\ov \g$ with the form $<\cdot,\cdot>$ defined by $<a,b>=(b,a)$ one obtains the vertex algebra  $F(\ov \g)$.
\par
We define the Casimir operator of $\g$ as $\Omega_\g=\sum_ix^ix_i$ if $\{x_i\}$ is  a basis 
of $\g$ and $\{x^i\}$  its dual basis w.r.t. $(\cdot,\cdot)$ (see \cite[pag. 85]{Kacsuper}). Since $\Omega_\g$  supercommutes with any element of  $U(\g)$, the generalized eigenspaces of its action on $\g$ are ideals in $\g$. Observe that $\Omega_\g$ is a symmetric operator: indeed
\begin{align*}
(\Omega_\g(a),b)&=\sum_i([x^i,[x_i,a]],b)=\sum_i-p(x^i,[x_i,a])([[x_i,a],x^i],b)\\&=\sum_ip(x^i)([a,x_i],[x^i,b])=\sum_ip(x^i)(a,[x_i,[x^i,b]])=(a,\Omega_\g(b)).
\end{align*}
Since $\Omega_\g$ is symmetric, the generalized eigenspaces provide an orthogonal decomposition of $\g$. Moreover, since $\s$ preserves the form, we have that $\s\circ \Omega_\g=\Omega_\g\circ \s$, hence $\s$ stabilizes the generalized eigenspaces. Since $\s$ is assumed to be indecomposable, it follows that that  
$\Omega_\g$ has a unique eigenvalue. Let $2g$ be such eigenvalue. If the form $(\cdot,\cdot)$ is normalized as in \cite[(1.3)]{KW}, then the number $g$ is called the dual Coxeter number of $\g$.
\begin{lemma} \label{central}If $\Omega_\g-2gI\ne 0$ then $g=0$. Moreover, in such a case, $\Omega_\g(\g)$ is a central ideal. 
\end{lemma}
\begin{proof}
Let $\g=\sum_{S=1}^k\g(S)$ be a  orthogonal decomposition in $(\cdot,\cdot)$-irreducible ideals. Clearly $\Omega_\g(\g(i))\subset \g(i)$, hence we can assume without loss of generality that $\g$ is $(\cdot,\cdot)$-irreducible. 

If $x$ is in the center of $\g$, then $x$ is orthogonal to $[\g,\g]$, so, if $\g=[\g,\g]$, then $\g$ must be centerless. In particular, in this case, $\g=\underline \g$ is a sum of simple ideals, but, being $(\cdot,\cdot)$-irreducible, it is simple. Since $\Omega_\g-2gI$ is nilpotent, we have that $(\Omega_\g-2gI)(\g)$ is a proper ideal of $\g$, hence $\Omega_\g=2gI$. 

Thus, if $\Omega_\g-2gI\ne 0$, we must have $\g\ne[\g,\g]$. 
Since $\g\ne[\g,\g]$, the form becomes degenerate when restricted to $[\g,\g]$ and its radical is contained in the center of $\g$. It follows that the center of $\g$ is nonzero. Clearly $\Omega_\g$ acts trivially on the center, hence $g=0$. 

Since $\g^{(2)}$ is an ideal of $\g$, clearly $\Omega_\g$ acts on it. Since $\Omega_\g(\mathfrak z(\g))=0$, this action descends to 
$\underline \g$. Recall that we have $\underline \g=\bigoplus\limits_{i=1}^k \underline \g(i)$, with $\underline \g(i)$ simple ideals. We already observed that these ideals are inequivalent as $\g_0$-modules, thus $\Omega_\g(\underline \g(i))\subset \underline \g(i)$. Since $\Omega_\g(\underline \g(i))$ is a proper ideal, we see that $\Omega_\g(\underline\g)=0$. Thus $\Omega_\g(\g^{(2)})\subset \mathfrak z(\g)\cap\g^{(2)}$. We now check that $\Omega_\g(M_{triv})=0$. Let $x\in M(\l)$. If $x_i\in \g^{(2)}_1$, by Lemma \ref{Mcommg2}, $[x_i,x]=0$. If  $x_i\in M(\mu)$ with $\mu\ne -\l$, then 
$[x_i,x]=0$. If $x_i\in M(-\l)$, then $[x^i,[x_i,x]]=(x_i,x)[x^i,h_\l]=(x_i,x)\Vert \l\Vert^2x_i=0$. It follows that $\Omega_\g(x)=\Omega_{\g_0}(x)=\Vert \l\Vert^2x=0$. The final outcome is that $\Omega_\g(\g_1)\subset \mathfrak z(\g)\cap\g^{(2)}\subset \h'$. Thus, since $\Omega_\g$ preserves parity, 
\begin{equation}\label{omegag1}
\Omega_\g(\g_1)=\{0\}.
\end{equation} It follows that $\Omega_\g(\g_0)=\Omega_\g(\g)$ is an ideal of $\g$ contained in $\g_0$. Since $\Omega_\g$ is nilpotent, $\Omega_\g(\g)$ is a nilpotent ideal, hence it intersects trivially $[\g_0,\g_0]$. It follows that  $[\Omega_\g(\g),\g_0]=0$. Since $ \Omega_\g(\g)$ is an ideal contained in $\g_0$, $[\Omega_\g(\g),\g_1]\subset\g_1\cap\g_0=\{0\}$ as well. The result follows.
\end{proof}
\begin{rem}\label{simple} Note that we have proved the following fact: if $\g$ is centerless and $(\cdot,\cdot)$-irreducible then it is simple (cf. \cite[Theorem 2.1]{Be}).
\end{rem}

Set $C_\g=\Omega_\g-2gI_\g$.
\begin{prop}\label{isomorfi} Assume $k+g\ne0$. Let   $\{x_i\}$ be a basis 
of $\g$ and let $\{x^i\}$ be its dual basis w.r.t. $(\cdot,\cdot)$.  For $x\in\g$ set  
\begin{equation}\label{xtilde}\widetilde x= x-
\frac{1}{2}\sum\limits_i :\overline{[x,x_i]}\overline x^i:+\frac{1}{4(k+g)}C_\g(x).\end{equation} 
The map $x\mapsto \widetilde  x$, $\ov
y\mapsto \ov y$,  induces an isomorphism of vertex 
algebras 
$V^k(\g)\otimes F(\overline \g)\cong V^{k+g,1}(\g)$.
\end{prop}
\begin{proof}Set $\a=\frac{1}{4(k+g)}$.
Fix $a,b\in\g$.
Since, by Lemma \ref{central}, $C_\g(b)$ is central, we get, from Wick formula \eqref{Wick},
that
\begin{align*} 
&[a_\l\widetilde b]=
[a_\l  
b]-\tfrac{1}{2}\sum_i[a_\l:\overline{[b,x_i]}\overline x^i:]+\a\l(a,C_\g(b))K\\
&=[a_\l b]-
\tfrac{1}{2}\sum_i(:[a_\l\overline{[b,x_i]}]\overline  
x^i:+p(a,\ov{[b,x_i]}):\overline{[b,x_i]}[a_\l\overline x^i]:)\\&-\frac{1}{2}\sum_i\int_0^\l[[a_\l 
\overline{[b,x_i]}]_\mu \overline x^i]d \mu+\a\l(a,C_\g(b))K\\
&=[a_\l b]-
\tfrac{1}{2}\sum_i(:\overline{[a,[b,x_i]]}\overline 
x^i:+p(a,\ov{[b,x_i]}):\overline{[b,x_i]}\,\overline{[a,x^i]}:)\\
&-\frac{1}{2}\sum_i\l(x^i,[a,[b,x_i]])\ov K+\a\l(a,C_\g(b))K.
\end{align*}
Using the invariance of the form and Jacobi identity, we have

\begin{align}\notag
[a_\l\widetilde b]&=[a_\l 
b]-
\tfrac{1}{2}\sum_i(:\overline{[a,[b,x_i]]}\overline  
x^i:-p(a,b):\overline{[b,[a,x_i]]}\overline x^i:)
\\\notag&-\frac{1}{2}\sum_i\l(x^i,[a,[b,x_i]])\ov K+\a\l(a,C_\g(b))K
\\\notag&=[a_\l  
b]-\tfrac{1}{2}\sum_i:\overline{[[a,b],x_i]}\overline x^i:
\\\notag&-\frac{1}{2}\sum_i\l(a,[x^i,[x_i,b]])\ov K+\a\l(a,C_\g(b))K
\\\label{pp}&=[a_\l  
b]-\tfrac{1}{2}\sum_i:\overline{[[a,b],x_i]}\overline x^i:
-\frac{1}{2}\l(a,\Omega_\g(b))\ov K+\a\l(a,C_\g(b))K\end{align} 

Next we prove that 
\begin{equation}\label{tre} [\overline a_\l 
\widetilde b]=0.
\end{equation} By Lemma \ref{central} and \eqref{Wick}, 
\begin{align*}[\overline a_\l \widetilde b]&=[\overline 
a_\l b]-
\tfrac{1}{2}\sum_i[\overline a_\l:\overline{[b,x_i]}\overline  
x^i:]\\&=p(b)\overline{[a,b]}-\tfrac{1}{2}\sum_i(
:[\overline a_\l 
\overline{[b,x_i]}]\overline x^i:+p(\ov a,\ov{[b,x_i]}):\overline{[b,x_i]}[\overline 
a_\l \overline x^i]:)\\
&=p(b)\overline{[a,b]}-\tfrac{1}{2}\sum_i(([b,x_i],a)\overline 
x^i+p(\ov a,\ov{[b,x_i]})(x^i,a)\overline{[b,x_i]})\\&=p(b)\overline{[a,b]}-
\tfrac{1}{2}(p(b)\overline{[a,b]}+p(b)\overline{[a,b]})=0.\end{align*} 
We now compute $[\widetilde a_\l \widetilde b]$. Using \eqref{tre}, we find $[\widetilde a_\l\widetilde  
b]=[a_\l\widetilde b]+\a[C_\g(a)_\l\widetilde b]$, hence, by \eqref{pp}
\begin{align*}\label{atbt}[\widetilde a_\l \widetilde b]&=[a_\l  
b]-\tfrac{1}{2}\sum_i:\overline{[[a,b],x_i]}\overline x^i:
-\tfrac{1}{2}\l(a,\Omega_\g(b))\ov K+\a\l(a,C_\g(b))K+\\
&+\a[C_\g(a)_\l\widetilde b]\\
&=[a_\l  
b]-\tfrac{1}{2}\sum_i:\overline{[[a,b],x_i]}\overline x^i:
-\tfrac{1}{2}\l(a,\Omega_\g(b))\ov K+\a\l(a,C_\g(b))K+\\
&+\a([C_\g(a)_\l  
b]-\tfrac{1}{2}\sum_i:\overline{[[C_\g(a),b],x_i]}\overline x^i:)\\
&
-\a(\frac{1}{2}\l(C_\g(a),\Omega_\g(b))\ov K+\a\l(C_\g(a),C_\g(b))K).
\end{align*}

By Lemma \eqref{central}, $[C_\g(a),b]=0$. Since $\Omega_\g(b)\in [\g,\g]$ and $C_\g(a)$ is central, we have $(C_\g(a),\Omega_\g(b))=0$.

 The term $(C_\g(a),C_\g(b))$ is zero as well: if $g\ne0$, then, by Lemma \ref{central}, $C_\g(b)=0$, and, if $g=0$, as above, $(C_\g(a),C_\g(b))=(C_\g(a),\Omega_\g(b))=0$. Thus, we can write
\begin{align*}
[\widetilde a_\l \widetilde b]&=[a_\l  
b]-\tfrac{1}{2}\sum_i:\overline{[[a,b],x_i]}\overline x^i:
-\tfrac{1}{2}\l(a,\Omega_\g(b))\ov K+\a\l(a,C_\g(b))K+\\
&+\a([C_\g(a),b]+\l(C_\g(a),b)K)\\
&=[a_\l  
b]-\tfrac{1}{2}\sum_i:\overline{[[a,b],x_i]}\overline x^i:
-\tfrac{1}{2}\l(a,\Omega_\g(b))\ov K+\a\l(a,C_\g(b))K+\\
&+\l(C_\g(a),b)K).
\end{align*}
In the last equality we used the fact that $C_\g(a)$ is central.

Since  $\Omega_\g$ (hence $C_\g$) is symmetric, we have
$$
[\widetilde a_\l \widetilde b]=[a_\l  
b]-\tfrac{1}{2}\sum_i:\overline{[[a,b],x_i]}\overline x^i:
-\tfrac{1}{2}\l(a,\Omega_\g(b))\ov K+2\a\l(a,C_\g(b))K.
$$
Note that $C_\g([a,b])=0$. In fact, for any $z\in\g$, 
$$
 (C_\g([a,b]),z)= ([a,b],C_\g(z))= (a,[b,C_\g(z)])=0.
$$
It follows that $[a_\l  
b]-\tfrac{1}{2}\sum_i:\overline{[[a,b],x_i]}\overline x^i:=[a,  
b]+\l(a,b)K-\tfrac{1}{2}\sum_i:\overline{[[a,b],x_i]}\overline x^i:=\widetilde{[a,b]}+\l(a,b)K$. Hence
\begin{align*}
[\widetilde a_\l \widetilde b]&=\widetilde{[a,b]}+\l(a,b)K-\tfrac{1}{2}\l(a,\Omega_\g(b))\ov K+2\a\l(a,C_\g(b))K\\
&=\widetilde{[a,b]}+\l(a,b)K-\l g(a,b)\ov K-\tfrac{1}{2}\l(a,C_\g(b)\ov K\\&+2\a\l(a,C_\g(b)K).\end{align*} 
Thus, in $V^{k+g,1}(\g)$, we have
\begin{align*}
[\widetilde a_\l \widetilde b]&=\widetilde{[a,b]}+\l(a,b)(k+g)-\l g(a,b)-\tfrac{1}{2}\l(a,C_\g(b))|0\rangle\\&+2\a\l(a,C_\g(b))(k+g)|0\rangle
\end{align*}
so, recalling that $\a=\frac{1}{4(k+g)}$, we get
$$
[\widetilde a_\l \widetilde b]=\widetilde{[a,b]}+\l k(a,b).
$$
We can now finish the proof as in Proposition 2.1 of \cite{KMP}.
 \end{proof}
 
 \begin{rem}
 If $g\ne 0$, as in Proposition 2.1 of \cite{KMP}, the map $x\mapsto \widetilde x$, $\ov y\mapsto \ov y$, ${\mathcal K}\mapsto K-g\ov K$, $\ov{\mathcal K}\mapsto \ov K$ defines a homomorphism of Lie conformal algebras $Cur(\g)\otimes F(\ov\g)\to V(R)$. In particular, if $g\ne 0$,
 Propositon \ref{isomorfi} holds true for any $k\in\C$. 
  \end{rem}

Let $A$ be a
vector superspace with a 
non-degenerate bilinear
skewsupersymmetric form $(\cdot,\cdot)$ and $\s$ an elliptic operator preserving the parity and leaving the form invariant. If $r\in \R$ let $\bar r=r+\ganz\in\R/\ganz$. Let $A^{\bar r}$  be the $e^{
2 \pi i r}$ eigenspace of $A$. 
We set 
$L(\s,A)=\oplus_{\mu\in\half+\bar r}(t^\mu\otimes A^{\bar r})$ and define 
the bilinear form 
$<\cdot,\cdot>$ on $L(\s,A)$ by setting $<t^\mu\otimes 
a,t^\nu\otimes  b>=\d_{\mu+\nu,- 1}(a,b)$. \par
If $B$ is any superspace endowed with a non-degenerate bilinear
skewsupersymmetric form $<\cdot,\cdot>$, we denote by 
$\mathcal{W}(B)$ be the quotient of the tensor algebra of $B$ modulo the ideal generated by 
$$a\otimes b-p(a,b)b\otimes a-<a,b>,\quad a,b\in B.$$
We now apply this construction to $B=L(\s,A)$ to obtain $\mathcal W(L(\s,A))$.
We choose   a 
maximal  isotropic subspace
$L^+$ of $L(\s,A)$ as follows: 
 fix a maximal isotropic subspace $A^+$  of 
$A^{\bar 0}$, and let 
$$L^+= \bigoplus_{\mu>-\frac{1}{2}}(t^\mu\otimes A^{\bar\mu}
)\oplus(t^{-
\frac{1}{2}}\otimes A^+ ).$$ \par We obtain a $\mathcal W(L(\s,A))$-module 
$CW(A)=\mathcal{W}(L(\s,A))/\mathcal{W}(L(\s,A))L^+$ (here $CW$ stands for ``Clifford-Weil").  

Note that $-I_A$ induces an involutive automorphism of $C(A)$ that we denote by $\omega$. Set $\tau=\omega\circ\s$. Then we can define fields 
$$Y(a,z)=\sum_{n\in\half+\bar r} (t^n\otimes a)z^{-n-1},\quad a\in\ 
A^{\bar r},
$$ 
where we let 
$t^n\otimes a$ act on $CW(A)$ by left multiplication. Setting furthermore 
$Y(\ov{\mathcal K},z)=I_A$, we get  a $\tau$-twisted representation of $C(A)$  on $CW(A)$ (that descends to a representation of $F(A)$). 
\par
Take now $A=\g$, and let  $\s$ be an automorphism of $\g$ as in Section \ref{2}.
Let $\g^{\ov 0}= \n^0\oplus \h^0\oplus \n^0_-$ be the triangular decomposition provided by Proposition \ref{triangular} applied to $\g^{\ov 0}$.
Choosing an isotropic subspace $\h^+$ of $\h^0$ we can choose the maximal isotropic subspace of $\g^{\ov 0}$ provided by Lemma \ref{isotropic} and 
construct the corresponding Clifford-Weil module $CW(\ov \g)$, which  we regard as a  $\tau$-twisted representation of $F(\ov \g)$. 

In light of Proposition \ref{isomorfi}, given a $\s$-twisted representation  $M$ of $V^k(\g)$, we can form $\s\otimes \tau$-twisted representation $X(M)=M\otimes CW(\ov\g)$ of $V^{k+g,1}(\g)$.

In the above setting, we choose $M$ in  a particular class of representations arising from the theory of twisted affine algebras.  We recall briefly their construction and refer the reader to \cite{KMP} for more details.

Let
$L'(\g,\si)=\sum_{j\in\R}(t^j\otimes\g^{\ov j})\oplus \C K$. This is a Lie 
superalgebra  with bracket defined by
$$ [t^m\otimes  a,t^n\otimes b ]=t^{m+n}\otimes[a,b] + 
\d_{m,-n}m(a,b)K,\quad m,n\in\R,
$$ $K$ being a central element.\par

\par  Let  $(\h^0)'=\h^0+\C K$. If 
$\mu\in ((\h^0)')^*$, we set $\ov\mu=\mu_{|\h^0}$. Set $\mathfrak 
n'=\mathfrak  n^0+\sum_{j>0}t^j\otimes\g^{\ov
j}$. Fix $\L\in ((\h^0)')^*$. A 
$L'(\g,\si)$-module $M$ is called a highest weight module with 
highest weight 
$\L$ if there is a nonzero vector $v_\L\in M$ such that 
\begin{equation}\label{highest}\mathfrak n'(v_\L)=0,\,\ 
hv_\L=\L(h)v_\L \text{  for $h\in(\h^0)'$, }\
U(L'(\g,\si))v_\L=M.\end{equation}

Let 
$\D^{\ov j}$ be the set of $\h^0$-weights of $\g^{\ov j}$.
 If $\mu\in(\h^0)^*$ and $\mathfrak m$ is any $\h^0$-stable subspace of $\g$, then we let $\mathfrak m_\mu$ be the corresponding weight space.  Denote by  $\D^0$  the 
set of roots (i. e. the nonzero $\h^0$-weights)  of $\g^{\ov 0}$. Set $\D^0_+=\{\a\in\D^0\mid \n^0_\a\ne\{0\}\}$.

Since $\n^0$ and $\n^0_-$ are non degenerately paired, we have that $-\D^0_+=\{\a\in\D^0\mid (\n_-)_\a\ne\{0\}\}$. By the decomposition \eqref{triang} we have 
$\D^0=\D^0_+\cup -\D^0_+$. 

Set
\begin{equation}\label{rhos}
\rho^{\ov 0}=\half\sum_{\a\in \D^0_+}(\sdim\,\n^{0}_\a) \a,\quad\rho^{\ov j}=
\half\sum_{\a\in\D^{\ov j}}(\sdim\,\g^{\ov j}_\a)\a\quad \text{if $\ov 
j\ne\ov 0$},
\end{equation}
\begin{equation}\label{rhosigma}
\rho_\si=
\sum_{0\le j\le\half}(1-2 j)\rho^{\ov j}.
\end{equation}
Finally set
\begin{equation}\label{zgs}
z(\g,\s)=\half\sum_{0\le  j< 
1}\frac{j(1-j)}{2}\sdim\,\g^{\ov j}.
\end{equation}
 Here and in the following we denote by $\sdim{}V$ the superdimension $\dim V_0 - \dim V_1$ of a superspace $V=V_0\oplus V_1$.
 
If $X$ is a twisted representation of a vertex algebra $V$  (see \cite[\S\ 3]{KMP}) and $a\in V^{\ov j}$, we let 
$$Y^X(a,z)=\sum_{n\in \ov j}a^X_{(n)}z^{-n-1}
$$
be the corresponding field. As explained in \cite{KMP}, a highest weight module $M$ for $L'(\g,\s)$ of highest weight $\L$ becomes automatically a $\s$-twisted representation of 
$V^{k}(\g)$ where $k=\L(K)$.

Set 
\begin{align}&\label{sugawcur}L^\g=\half\sum_i:x^ix_i:\in
V^{k}(\g),\\\label{sugawferm}&L^{\ov\g}=\half\sum_i:T(\ov x_i)\ov x^i:\in
F(\ov\g).\end{align}
 We can now prove (cf. \cite[Lemma 3.2]{KMP}):

\begin{lemma}\label{xixi} If $M$ is a highest weight module for $L'(\g,\s)$  with highest weight $\L$ and $\L(K)=k$
then
\begin{align}\label{azioneasinistra}
&\sum_i:C_\g(x^i)x_i:_{(1)}^M(v_\L)=\ov\L(C_\g(h_{\ov\L}))v_\L\\
&(L^\g)^M_{(1)}(v_\L)=\half(\ov \L+2\rho_\s,\ov \L)v_\L+kz(\g,\s)v_\L.
\end{align} 
\end{lemma}

\begin{proof}  We can and do choose $\{x_i\}$ so that 
$x_i\in\g^{\ov s_i}$, for some $ \ov 
s_i\in\R/\ganz$. Let $A$ be any parity preserving operator on $\g$ which commutes with $\s$. In particular $A$ preserves $\g^{\ov j}$ for any $j\in\R$. By
(3.4) of \cite{KMP}, we have
\begin{align}\notag
\sum_i  :A(x^{i})x_{i}:^M_{(1)}&=\sum_i\left(\sum_{n<-s_i} 
A(x^{i})^M_{(n)} 
(x_{i})^M_{(-n)}+\sum_{n\geq -s_i}p(x_i) (x_{i})^M_{(-n)} 
A(x^{i})^M_{(n)}\right)\\&- 
\sum_{r\in\ganz_+}\binom{-s_i}{r+1}(A(x^{i})_{(r)}(x_{i}))^M_{(- 
r)}.\label{Sugawaraexpansion}
\end{align}
We choose $s_i\in [0,1)$, 
thus
\begin{align*}
&\sum_i :A(x^{i})x_{i}:^M_{(1)}(v_\L)=\\&\sum_i(p(x_i)(x_{i})^M_{( 
s_i)}A (x^{i})^M_{(-s_i)}+
s_i[A(x^{i}),x_{i}]^M_{(0)}-k\binom{-s_i}{2}(A(x^i),x_i))(v_\L),\end{align*}
which we can rewrite as 
\begin{align*}
&\sum_i(A (x^{i})^M_{(-s_i)}(x_{i})^M_{( 
s_i)}+
(s_i-1)[A(x^{i}),x_{i}]^M_{(0)}-k(\binom{-s_i}{2}-s_i)(A(x^i),x_i))(v_\L)=\\
&\sum_{i:s_i=0}A (x^{i})^M_{(0)}(x_{i})^M_{( 
0)}(v_\L)+\sum_i(
(s_i-1)[A(x^{i}),x_{i}]^M_{(0)}-k\binom{s_i}{2}(A(x^i),x_i))(v_\L).\end{align*}\par 

Assume now that $A=C_\g$. Then, since $C_\g(x^i)$ is central, $[C_\g(x^{i}),x_{i}]=0$. Note also that $\sum_{i:s_i=j}(C_\g(x^i),x_i)$ is the supertrace of $(C_\g)_{|\g^{\ov j}}$. Since $C_\g$ is nilpotent, we obtain that $\sum_{i:s_i=j}(C_\g(x^i),x_i)=0$. Thus
\begin{align*}
\sum_i :C_\g(x^{i})x_{i}:^M_{(1)}(v_\L)=\sum_{i:s_i=0}C_\g (x^{i})^M_{(0)}(x_{i})^M_{( 
0)}(v_\L).
\end{align*}

We choose the basis $\{x_i\}$ by choosing, for each $\a\in\D^0\cup\{0\}$, a basis $\{(x_\a)_i\}$ of $\g^{\ov 0}_\a$. Set $\{x_\a^i\}$ to be its dual basis in $(\g)_{-\a}$. If $x\in\g_\a$, then $0=C_\g([h,x_{\a}])=\a(h)C_\g(x_\a)$. If $\a\ne 0$ then this implies $C_\g(x_\a)=0$. If $\a=0$ and $(x_\a)_i\in\g_1$, then, by \eqref{omegag1}, we have that $C_\g((x_\a)_i)=0$ as well. This implies that, if $\{h_i\}$ is an orthonormal basis of $\h^0$,
\begin{align*}
\sum_i :C_\g(x^{i})x_{i}:^M_{(1)}(v_\L)&=\sum_{i}C_\g (h_{i})^M_{(0)}(h_{i})^M_{( 
0)}(v_\L)\\&=\sum_{i}\L(C_\g (h_{i}))\L(h_{i})v_\L=\ov\L(C_\g(h_{\ov \L}))v_\L.
\end{align*}

Let now $A=Id$. Clearly we can assume that the basis $\{(x_\a)_i\}$ of $\g^{\ov 0}_\a$ is the union of a basis of $\n_\a^0$ and a basis of $(\n^0_-)_\a$ if $\a\in\D^0_+$, while, if $\a=0$, we can choose the basis $\{(x_\a)_i\}$ to be the union of a basis of $\n_\a^0$, a basis of $(\n^0_-)_\a$ and an orthonormal basis $\{h_i\}$ of $\h^0$. We can therefore write
$$\sum\limits_{i: s_i=0}x^i_{(0)}(x_i)_{(0)}=2(h_{\rho_0})_{(0)}+\sum_i 
(h_i)_{(0)}^2+2\sum_{(x_\a)_i\in\n^0_\a}(x_{\a}^i)_{(0)}((x_\a)_i)_{(0)}.
$$ We  find that
\begin{align*}\sum_i 
:x^{i}x_{i}:^M_{(1)}(v_\L)&=(\ov\L+2\rho_0,\ov\L)v_\L+k(\sum_{0<j<1}
\frac{j(1-j)}{2}\sdim\,\g^{\ov j})v_\L\\ &+
\sum_{i:s_i>0}s_i[x^{i},x_{i}]^M_{(0)}(v_{\L}).\end{align*}\par In 
order to  evaluate
$\sum_{i:s_i>0}s_i[x^{i},x_{i}]^M_{(0)}(v_{\L})$, we observe  that
$$\sum_{i:s_i=s}[x^i,x_i]=\sum_{i:s_i=1-s}p(x_i)[x_i,x^i]=-\sum_{i:s_i=1-s}[x^i,x_i].$$ This relation 
is easily  derived by exchanging the roles
of $x_i$ and $x^i$.  
Hence\begin{align*}\sum_{i:s_i>0}s_i[x^{i},x_{i}]&=
\sum_{i: \frac{1}{2}>s_i>0}s_i[x^{i},x_{i}]+\sum_{i: 
1>s_i>\frac{1}{2}}s_i[x^{i},x_{i}]\\&=\sum_{i: 
\frac{1}{2}>s_i>0}s_i[x^{i},x_{i}]+\sum_{i: 0<s_i<\frac{1}{2}}(s_i-1)[x^{i},x_{i}]\\&=-
\sum_{i: \frac{1}{2}>s_i>0}(1- 
2s_i)[x^{i},x_{i}].\end{align*}\par We can choose
$x_i$ in 
$\g^{\ov s_i}_\alpha$ so  that $[x^i,x_i]=-p(x_i)h_{\alpha}$, hence 
\begin{equation}\label{sumsixixi}
\sum_{i}s_i[x^{i},x_{i}]=\sum_{i:0<s_i<\frac{1}{2}}2(1-
2s_i)h_{\rho^{\ov s_i}}, 
\end{equation}
hence
$$
\sum_{i}s_i[x^{i},x_{i}]^M_{(0)}(v_{\L})=\sum_{i:0<s_i<\frac{1}{2}}(1-
2s_i)(2\rho^{\ov s_i},\ov\L)v_\L. $$
This completes the proof of 
\eqref{azioneasinistra}.\end{proof}
\section{Dirac operators}\label{3}
As in the previous section $\{x_i\}$
is a homogeneous basis  of $\g$ and  $\{x^i\}$ is its dual basis. 
The  following element  of 
$V^{k,1}(\g)$:
\begin{equation}\label{G}G_\g=\sum_i:x^i\ov x_i:-\frac{1}{3}\sum_{i,j}:\ov{[x^i,x_j]}\ov x^j\ov x_i:
\end{equation} 
is called the affine Dirac operator. It has the following  properties:
\begin{align} &[a_\la G_\g]=\la k\overline a,\quad[\overline a_\la G_\g]=a\label{abarG}.
\end{align}
By the  sesquilinearity of the
$\l$-bracket, we also have
\begin{align} &[(G_\g)_\l a]=p(a)k(\la \overline a+T(\ov a)),\quad[(G_\g)_\l \ov a]=p(a)a\label{Gabar}.
\end{align}
If $\g$ is purely even, $G_\g$ is the Kac-Todorov-Dirac field considered in \cite{KMP} and in \cite{KacD} as an analogue of Kostant's cubic 
Dirac operator. It can be proved,  by using a suitable Zhu functor $\pi: V^{k,1}(\g)\to U(\g)\otimes \mathcal W(\ov \g)$, that $\pi(G_\g)$ is the Dirac operator considered 
by Huang and Pandzic in \cite{HP2}.

Write for shortness $G$ instead of $G_\g$. We 
want to
calculate $[G_\l G]$.    
We proceed in steps.     
Set 
\begin{equation}\label{teta}\theta(x)=\half \sum_i:\ov{[x,x_i]}\ov x^i:,\end{equation} and
note  that
\begin{equation}\label{diractilde} G=\sum_i:x^i\overline 
x_i:-
\tfrac{2}{3}\sum_{i}:\theta(x^i)\overline 
x_i:.
\end{equation}

We start by collecting some formulas.
\begin{lemma}
\begin{equation}
[a_\l\theta(b)]=\theta([a,b])+\frac{1}{2}\l(a,\Omega_\g(b)),\label{altheta1}
\end{equation}
\begin{equation}
[\ov a_\l\theta(b)]=p(b)\ov{[a,b]},\label{altheta2}
\end{equation}
\begin{equation}
[\theta(a)_\l b]=\theta([a,b])+\frac{1}{2}\l(a,\Omega_\g(b)),\label{altheta3}
\end{equation}
\begin{equation}
[\theta(a)_\l\ov b]=\ov{[a,b]}\label{altheta4},
\end{equation}
\begin{equation}
[\theta(a)_\l \theta(b))]=\theta([a,b])+\frac{1}{2}\l(a,\Omega_\g(b)),\label{altheta6}
\end{equation}
\begin{equation}
[\theta(a)_\l \sum_i:x^i\ov x_i:]=-\sum_i:(x^i-\theta(x^i))\ov{[x_i,a]}:+\frac{3}{2}\l\ov{\Omega_\g(a)},\label{altheta5}
\end{equation}
\begin{equation}
[\theta(a)_\l \sum_i:\theta(x^i)\ov x_i:]=\frac{3}{2}\l\ov{\Omega_\g(a)},\label{altheta7}
\end{equation}
\begin{equation}
\sum_i:x^i\theta(x_i):=\sum_i:\theta(x^i)x_i:,\label{altheta8}
\end{equation}
\begin{equation}
\sum_i:\theta(x^i)\theta(x_i):=0.\label{altheta9}
\end{equation}
\end{lemma}
\begin{proof}Formulas \eqref{altheta1} and \eqref{altheta2} have been proven in the proof of Proposition \ref{isomorfi}. Formulas \eqref{altheta3} and \eqref{altheta4} are obtained by applying   sesquilinearity of the
$\l$-bracket to \eqref{altheta1} and \eqref{altheta2}. From \eqref{altheta4} and \eqref{seconda} one derives that 
\begin{equation}\label{tildelbar}
[\ov a_\l (b-\theta(b))]=[(b-\theta(b))_\l\ov a]=0,
\end{equation}
hence $[\theta(a)_\l (b-\theta(b))]=0$. This implies \eqref{altheta6}.
 Using Wick's formula \eqref{Wick},  \eqref{altheta3}, and \eqref{altheta4} we get
$$
[\theta(a)_\l \sum_i:x^i\ov x_i:]=\sum_i(:\theta([a,x^i])\ov x_i:+p(a,x_i):x^i\ov{[a,x_i]}:)+\frac{3}{2}\l\ov{\Omega_\g(a)}.
$$
Note that 
\begin{equation}\label{thetatheta}\sum_i:\theta(x^i)\ov{[x_i,a]}:=\sum_i:\theta([a,x^i])\ov x_i:
\end{equation} so \eqref{altheta5} follows. Likewise 
$$
[\theta(a)_\l \sum_i:\theta(x^i)\ov x_i:]=\sum_i(:\theta([a,x^i])\ov x_i:+p(a,x_i):\theta(x^i)\ov{[a,x_i]}:)+\frac{3}{2}\l\ov{\Omega_\g(a)}.
$$
 so, by \eqref{thetatheta}, \eqref{altheta7} follows as well. For \eqref{altheta8}, it is enough to apply formula (1.39) of \cite{KacD} and \eqref{altheta1}. Finally, 
\begin{align*}
\sum_i:\theta(x^i)\theta(x_i):&=\tfrac{1}{2}\sum_{i,r}:\theta(x^i):\ov{[x_i,x_r]}\ov x^r::=\tfrac{1}{2}\sum_{i,r,s}:\theta(x^i)([x_i,x_r],x^s):\ov x_s\ov x^r::\\
&=\tfrac{1}{2}\sum_{i,r,s}:\theta(x^i)(x_i,[x_r,x^s]):\ov x_s\ov x^r::=\tfrac{1}{2}\sum_{r,s}:\theta([x_r,x^s]):\ov x_s\ov x^r::
\end{align*}
and
\begin{align*}
\sum_{r,s}:\theta([x_r,x^s]):\ov x_s\ov x^r::&=\sum_{r,s}-p(x_r,x^s)p(\ov x_s, \ov x_r):\theta([x^s,x_r]):\ov x^r\ov x_s::\\
&=-\sum_{r,s}p(x^s)p( x_r):\theta([x^s,x_r]):\ov x^r\ov x_s::\\
&=-\sum_{r,s}:\theta([x_s,x^r]):\ov x_r\ov x^s::.
\end{align*}
so \eqref{altheta9} holds.
\end{proof}

We start our computation of $[G_\l G]$. First observe that, by \eqref{Gabar},
\begin{equation}\label{[Gxx]}\	\sum_i[G{}_\la  :x^i\ov x_i:]=\sum_i :
x^i  x_i:+k\sum_i:T(\overline
x_i)\overline  x^i:+k\frac{\la^2}{2}\sdim\g.
\end{equation} 
Next we compute $[G_\l \sum_i :\theta(x^i)\ov x_i:]$. By \eqref{diractilde}, \eqref{altheta5}, and \eqref{altheta7},
$$
[\theta(a)_\l G]=-\sum_i:(x^i-\theta(x^i))\ov{[x_i,a]}:+\frac{1}{2}\l\ov{\Omega_\g(a)}
$$
and, by sesquilinearity,
\begin{equation}\label{Gltheta2}
[G_\l \theta(a) ]=p(a)\sum_i:(x^i-\theta(x^i))\ov{[x_i,a]}:+p(a)\frac{1}{2}(\l+T)(\ov{\Omega_\g(a)})
\end{equation}
By Wick's formula, \eqref{Gltheta2}, and \eqref{Gabar},
\begin{align*}
&[G_\l \sum_i :\theta(x^i)\ov x_i:]=\sum_{i,j}::(x^j-\theta(x^j))\ov{[x_j,x_i]}:\ov x^i:)\\
&+\frac{1}{2}\sum_ip(x^i)(\l+T):\ov{\Omega_\g(x^i)}\ov x_i:+ \sum_i:\theta(x^i) x_i:\\
&+\sum_{i,j}\int_0^\l [:(x^j-\theta(x^j))\ov{[x_j,x_i]}:_\mu\ov x^i]d\mu\\
&+\frac{1}{2}\sum_ip(x^i)\int_0^\l [(\l+T)(\ov{\Omega_\g(x^i)})_\mu\ov x_i]d\mu).
\end{align*}

Let us compute the terms of the above sum one by one. By formula  (1.40) of \cite{KacD}, and \eqref{tildelbar} above, we have
\begin{align*}
&\sum_{i,j}::(x^j-\theta(x^j))\ov{[x_j,x_i]}:\ov x^i:
=\sum_{i,j}:(x^j-\theta(x^j)):\ov{[x_j,x_i]}\ov x^i::\\
&+\sum_{i,j}\int_0^Td\l:(x^j-\theta(x^j))[\ov{[x_j,x_i]}_\l \ov x^i]:\\
&=2\sum_j:(x^j-\theta(x^j))\theta(x_j):-\sum_{i,j}\int_0^Td\l:(x^j-\theta(x^j))([x^i,x_i],x_j):\\
&=2\sum_j:(x^j-\theta(x^j))\theta(x_j):
\end{align*}
Note that $\sum_ip(x^i):\ov{\Omega_\g(x^i)}\ov x_i:=0$. Indeed 
\begin{align*}
\sum_ip(x_i):\ov{\Omega_\g(x^i)}\ov x_i:&=\sum_i:\ov{\Omega_\g(x_i)}\ov x^i:=\sum_{i,r}(\Omega_\g(x_i),x^r):\ov x_r\ov x^i:\\
&=\sum_{i,r}(x_i,\Omega_\g(x^r)):\ov x_r\ov x^i:=\sum_{r}:\ov x_r\ov{\Omega_\g(x^r)}:\\
&=\sum_{r}p(\ov x_r):\ov{\Omega_\g(x^r)}\ov x_r:=-\sum_{r}p( x_r):\ov{\Omega_\g(x^r)}\ov x_r
\end{align*}
Using formula (1.38) of \cite{KacD} and \eqref{tildelbar} above we see that
$$
\sum_{i,j}\int_0^\l [:(x^j-\theta(x^j))\ov{[x_j,x_i]}:_\mu\ov x^i]d\mu=\sum_{i,j}\int_0^\l (x^j-\theta(x^j))(x_j,[x^i,x_i])d\mu=0
$$
Finally, since $\Omega_\g-2gId$ is nilpotent, it has zero supertrace, hence $$\sum_ip(x^i)\int_0^\l [(\l+T)(\ov{\Omega_\g(x^i)})_\mu\ov x_i]d\mu=\sum_ip(x^i)\int_0^\l \l(x_i,\Omega_\g(x^i))d\mu=g\l^2 \sdim\g.
$$

It follows that
\begin{align*}
[G_\l \sum_i :\theta(x^i)\ov x_i:]=2\sum_j:(x^j-\theta(x^j))\theta(x_j):+ \sum_i:\theta(x^i) x_i:+\frac{g}{2}\l^2 \sdim\g\end{align*}
Using \eqref{altheta8} and \eqref{altheta9}, we can conclude that 
$$
[G_\l \sum_i :\theta(x^i)\ov x_i:]=3\sum_i:x^i\theta(x_i):- \frac{3}{2}\sum_i:\theta(x^i)\theta(x_i):+\frac{g}{2}\l^2 \sdim\g,
$$
Combining this with \eqref{[Gxx]}, the final outcome is that
\begin{equation}\label{[GG]}[G{}_\la  G]=\sum_i :( x^i-\theta(x^i))
(x_i-\theta(x_i)):+k\sum_i:T(\overline
x_i)\overline  x^i:+\frac{\la^2}{2}(k-\frac{2g}{3})\sdim\g.
\end{equation}

Since $x-\theta(x)=\tilde x-\tfrac{1}{4k}C_\g(x)=\tilde x-\tfrac{1}{4k}\widetilde{C_\g(x)}$ we have that
\begin{align*}
\sum_i :( x^i-\theta(x^i))
(x_i-\theta(x_i)):&=\sum_i :(\tilde x^i-\tfrac{1}{4k}C_\g(x^i))
(\tilde x_i-\tfrac{1}{4k}C_\g(x_i)):\\
&=\sum_i :\tilde x^i\tilde x_i:-\tfrac{1}{2k}\sum_i:\widetilde{C_\g(x^i)}\tilde x_i:.
\end{align*}
We used the fact, that, since $C_\g$ is symmetric, $\sum_i:\widetilde{C_\g(x^i)}\tilde x_i:=\sum_i:\tilde x^i\widetilde{C_\g(x_i)}:$ and, since $C_\g^2=0$, $\sum_i:\widetilde{C_\g(x^i)}\widetilde{C_\g(x_i)}:=0$.
Thus \eqref{[GG]} can be rewritten as 
\begin{equation}\label{GGtilde}[G_\g{}_\la  G_\g]=\sum_i (:\tilde 
x^i\tilde  x_i:-\tfrac{1}{2k}:\widetilde{C_\g(x^i)}\tilde x_i:+k:T(\overline
x_i)\overline  x^i:)+\frac{\la^2}{2}(k-\frac{2g}{3})\sdim\g.
\end{equation}
Identifying 
$V^{k,1}(\g)$ with $V^{k-g}(\g)\otimes F(\ov \g)$ we  have that 
 \eqref{GGtilde}
 can be rewritten as
\begin{equation}\label{gquad}[G_\g{}_\la  
G_\g]=2L^\g\otimes\va-\tfrac{1}{2k}\sum_i:C_\g(x^i) x_i:\otimes|0\rangle+2k\va\otimes
L^{\ov\g}+\tfrac{\la^2}{2}(k-\tfrac{2}{3}g)\sdim{}\g.
\end{equation}

Recall that, given a highest weight representation $M$ of $L'(\g,\s)$, we constructed a $\s\otimes\tau$-twisted representation $X=X(M)$ of $V^{k,1}(\g)$. Setting $(G_\g)^X_n= (G_\g)^X_{(n+1/2)}$, we can write the field $Y^X(G_\g,z)$
as 
$$
Y^X(G_\g,z)=\sum_{n\in\ganz}G^X_n z^{-n-\frac{3}{2}}.
$$

Using the fact that $(G^X_0)^2=\tfrac{1}{2}[G^X_0,G^X_0]$ and
\eqref{gquad},  we have 
\begin{equation}(G^X_0)^2\!=\!(L^\g\!-\tfrac{1}{4k}\sum_i:C_\g(x^i) x_i\!:)^M_{(1)}\otimes I_{CW(\ov\g)}- 
kI_M\otimes
(L^{\ov\g})^{CW(\ov\g)}_{(1)}-\!\!
\tfrac{1}{16}(k-\tfrac{2}{3}g)(\sdim\g)\, I_X\label{gzeronuovo}.
\end{equation}

From now on we will write $a_{(n)}$ instead of $a^V_{(n)}$ when there is no risk of confusion for the twisted representation $V$.
\begin{lemma}\label{thetadot1} In $CW(\ov \g)$, if $x\in\g^{\ov s}$ and $n>0$, we have $\theta(x)_{(n)}\cdot1=0$.
\end{lemma}
\begin{proof}
Choose the basis $\{x_i\}$ of $\g$, so that $x_i\in\g^{\ov s_i}$. We can clearly assume $s,s_i\in[0,1)$.
We apply formula (3.4) of \cite{KMP} to get
\begin{align*}
\theta(x)_{(n)}&\!=\!\!\sum_{i,m< s+s_i-\tfrac{1}{2}}\ov{[x,x_i]}_{(m)}\ov x^i_{(n-m-1)}-p(x,x_i)p(x)p(x_i)\!\!\sum_{i,m\ge s+s_i-\tfrac{1}{2}}\ov x^i_{(n-m-1)}\ov{[x,x_i]}_{(m)}\\&-\sum_i\binom{s+s_i-\tfrac{1}{2}}{1}[\ov{[x,x_i]}_{(0)}\ov x^i]_{(n-1)}.
\end{align*}
If $m< s+s_i-\tfrac{1}{2}$, then $n-m-1>n-s-s_i-\tfrac{1}{2}$. Since $n\in \ov s$, $n>0$ and $s\in[0,1)$, we have $n-s\ge 0$. Thus $n-m-1>-s_i-\tfrac{1}{2}$. Since $s_i\in[0,1)$ and $n-m-1\in -\ov s_i+\tfrac{1}{2}$, we see that $n-m-1\ge-s_i+\tfrac{1}{2}>-\tfrac{1}{2}$. It follows that 
$\ov x^i_{(n-m-1)}\cdot 1=0$. If $m>s+s_i-\tfrac{1}{2}$, then $m>-\tfrac{1}{2}$ so $\ov{[x,x_i]}_{(m)}\cdot 1=0$. Since $[\ov{[x,x_i]}_{(0)}\ov x^i]=(x^i,[x,x_i])|0\rangle$, we see that, since $n>0$, $[\ov{[x,x_i]}_{(0)}\ov x^i]_{(n-1)}=0$.
We therefore obtain that
\begin{align*}
\theta(x)_{(n)}&=-p(x,x_i)p(x)p(x_i)\sum_{i}\ov x^i_{(n-s-s_i-\tfrac{1}{2})}\ov{[x,x_i]}_{(s+s_i-\tfrac{1}{2})}.
\end{align*}
If $s>0$ or $s_i>0$ then $\ov{[x,x_i]}_{(s+s_i-\tfrac{1}{2})}\cdot 1=0$, so we can assume $s=0$ and get that
\begin{align*}
\theta(x)_{(n)}&=-p(x,x_i)p(x)p(x_i)\sum_{i:s_i=0}\ov x^i_{(n-\tfrac{1}{2})}\ov{[x,x_i]}_{(-\tfrac{1}{2})}=-\sum_{i:s_i=0}\ov{[x,x_i]}_{(-\tfrac{1}{2})}\ov x^i_{(n-\tfrac{1}{2})}.
\end{align*}
Observing that, since $n>0$, $\ov x^i_{(n-\tfrac{1}{2})}\cdot 1=0$ we get the claim.
\end{proof}

\begin{lemma}\label{kostant} In $CW(\ov \g)$ we have that
\begin{align*}
\sum_{i,j:s_i=s_j=0}p(\ov{[x^i,x_j]},\ov x^j)(\ov x^j)_{(-\tfrac{1}{2})}(\ov{[x^i,x_j]})_{(-\tfrac{1}{2})}(\ov x_i)_{(-\tfrac{1}{2})}\cdot1=6(\ov h_{\rho_0})_{(-\tfrac{1}{2})}\cdot1.
\end{align*}
\end{lemma}
\begin{proof}Clearly we can choose the basis $\{x_i\}$ of $\g^{\ov 0}$ to be homogeneous with respect to the triangular decomposition $\g^{\ov 0}=\n^0\oplus\h^0\oplus \n^0_-$. We can also assume that the $x_i$ are $\h^0$-weight vectors. Let $\mu_i$ be the weight of $x_i$. Set $\b^0=\h^0\oplus \n^0$ and $\b^0_-=\h^0\oplus \n^0_-$. Then
\begin{align*}
\sum_{i,j:s_i=s_j=0}&p(\ov{[x^i,x_j]},\ov x^j)(\ov x^j)_{(-\tfrac{1}{2})}(\ov{[x^i,x_j]})_{(-\tfrac{1}{2})}(\ov x_i)_{(-\tfrac{1}{2})}\cdot1\\
&=\sum_{i,j:x_i\in\b^0_-,s_j=0}p(\ov{[x^i,x_j]},\ov x^j)(\ov x^j)_{(-\tfrac{1}{2})}(\ov{[x^i,x_j]})_{(-\tfrac{1}{2})}(\ov x_i)_{(-\tfrac{1}{2})}\cdot1\\
&=\sum_{i,j:x_i\in\b^0_-,s_j=0}(\ov{[x^i,x_j]})_{(-\tfrac{1}{2})}(\ov x^j)_{(-\tfrac{1}{2})}(\ov x_i)_{(-\tfrac{1}{2})}\cdot1\\
&+
\sum_{i,j:x_i\in\b^0_-,s_j=0}p(\ov{[x^i,x_j]},\ov x^j)([x^i,x_j], x^j)(\ov x_i)_{(-\tfrac{1}{2})}\cdot1
\end{align*}
Since $[x_j,x^j]\in\h^0$, we have that
\begin{align*}
&\sum_{i,j:x_i\in\b^0_-,s_j=0}p(\ov{[x^i,x_j]},\ov x^j)([x^i,x_j], x^j)(\ov x_i)_{(-\tfrac{1}{2})}\cdot1=\\&-\sum_{i,j:s_i=s_j=0}(x^i,[x^j, x_j])(\ov x_i)_{(-\tfrac{1}{2})}\cdot1=0.
\end{align*}
It follows that
\begin{align*}
\sum_{i,j:s_i=s_j=0}&p(\ov{[x^i,x_j]},\ov x^j)(\ov x^j)_{(-\tfrac{1}{2})}(\ov{[x^i,x_j]})_{(-\tfrac{1}{2})}(\ov x_i)_{(-\tfrac{1}{2})}\cdot1\\
&=\sum_{i,j:x_i\in\b^0_-,s_j=0}(\ov{[x^i,x_j]})_{(-\tfrac{1}{2})}(\ov x^j)_{(-\tfrac{1}{2})}(\ov x_i)_{(-\tfrac{1}{2})}\cdot1\\
&=\sum_{i,j:x_i,\in\b^0_-,s_j=0}p(\ov x_i, \ov x_j)(\ov{[x^i,x_j]})_{(-\tfrac{1}{2})}(\ov x_i)_{(-\tfrac{1}{2})}(\ov x^j)_{(-\tfrac{1}{2})}\cdot1\\
&+\sum_{i:x_i\in\b^0_-}(\ov{[x^i,x_i]})_{(-\tfrac{1}{2})}\cdot1
\end{align*}
Since $\sum_{i:x_i\in\b^0_-}(\ov{[x^i,x_i]})_{(-\tfrac{1}{2})}\cdot1=\sum_{i:x_i\in\n_-}p(x_i)(\ov h_{-\mu_i})_{(-\tfrac{1}{2})}\cdot1=2(\ov h_{\rho_0})_{(-\tfrac{1}{2})}\cdot1$, we need only to check that
\begin{equation}\label{hrho}
\sum_{i,j:x_i,\in\b^0_-,s_j=0}p(\ov x_i, \ov x_j)(\ov{[x^i,x_j]})_{(-\tfrac{1}{2})}(\ov x_i)_{(-\tfrac{1}{2})}(\ov x^j)_{(-\tfrac{1}{2})}\cdot1=4(\ov h_{\rho_0})_{(-\tfrac{1}{2})}\cdot 1.
\end{equation}
Now
\begin{align*}
&\sum_{i,j:x_i,\in\b^0_-,s_j=0}p(\ov x_i, \ov x_j)(\ov{[x^i,x_j]})_{(-\tfrac{1}{2})}(\ov x_i)_{(-\tfrac{1}{2})}(\ov x^j)_{(-\tfrac{1}{2})}\cdot1=\\
&\sum_{i,j:x_i,\in\b^0_-,s_j=0}p( x_i)(\ov{[x^j,x^i]})_{(-\tfrac{1}{2})}(\ov x_i)_{(-\tfrac{1}{2})}(\ov x_j)_{(-\tfrac{1}{2})}\cdot1=\\
&\sum_{i,j:x_i,x_j\in\b^0_-}p( x_i)(\ov{[x^j,x^i]})_{(-\tfrac{1}{2})}(\ov x_i)_{(-\tfrac{1}{2})}(\ov x_j)_{(-\tfrac{1}{2})}\cdot1=\\
&\sum_{i,j:x_i,x_j\in\b^0_-}p( x_i)(\ov x_i)_{(-\tfrac{1}{2})}(\ov x_j)_{(-\tfrac{1}{2})}(\ov{[x^j,x^i]})_{(-\tfrac{1}{2})}\cdot1\\
&+\sum_{i,j:x_i,x_j\in\b^0_-}p( x_i)( x_i,[x^j,x^i])(\ov x_j)_{(-\tfrac{1}{2})}\cdot1\\
&+\sum_{i,j:x_i,x_j\in\b^0_-}p( x_i)p(\ov{[x^j,x^i]},\ov x_i)(\ov x_i)_{(-\tfrac{1}{2})}( x_j,[x^j,x^i])\cdot1.
\end{align*}
Since $[x^j,x^i]\in \b^0_-$ only when $x^j,x^i\in\h^0$, we see that
$$\sum_{i,j:x_i,x_j\in\b^0_-}p( x_i)(\ov x_i)_{(-\tfrac{1}{2})}(\ov x_j)_{(-\tfrac{1}{2})}(\ov{[x^j,x^i]})_{(-\tfrac{1}{2})}\cdot1=0.
$$
Moreover both 
$$\sum_{i,j:x_i,x_j\in\b^0_-}p( x_i)( x_i,[x^j,x^i])(\ov x_j)_{(-\tfrac{1}{2})}\cdot1
$$ and
$$\sum_{i,j:x_i,x_j\in\b^0_-}p( x_i)p(\ov{[x^j,x^i]},\ov x_i)(\ov x_i)_{(-\tfrac{1}{2})}( x_j,[x^j,x^i])\cdot1
$$ 
are equal to $\sum_{x_i\in\b^0_-}(\ov{ [x^i,x_i]})_{(-\tfrac{1}{2})}\cdot1=2(\ov h_{\rho_0})_{(-\tfrac{1}{2})}\cdot1$. This proves \eqref{hrho}, hence the statement.

\end{proof}
\section{The very strange formula}
 We are interested in calculating 
$G^X_0(v_\Lambda\otimes 1)$, $v_\Lambda$ being a highest weight vector 
of a 
$L'(\g,\s)$-module $M$ with highest weight $\L$ such that $\L(K)=k-g$.  

Since $\s$ preserves the form $(\cdot,\cdot)$, we have that $\s\Omega_\g=\Omega_\g\s$. It follows that $\Omega_\g$ stabilizes $\g^{\ov j}$ for any $j$. Recall, furthermore, that $\Omega_\g(\g)$ is contained in the radical of the form restricted to $[\g,\g]$. In particular $\Omega_\g(\g)\subset \h$. We can therefore choose the maximal isotropic subspace $\h^+$ of $\h^0$ so that $\Omega_\g(\g^{\ov 0})\subset \h^+$. With this choice we are now ready to prove the following result.

\begin{prop}\label{azioneg0}  \begin{equation}\label{0} G^X_0(v_\Lambda\otimes 1)=v_\Lambda\otimes 
(\ov h_{\ov\Lambda+\rho_\s})_{(-\tfrac{1}{2})}\cdot  1.
\end{equation}
\end{prop}
\begin{proof} Since $C_\g$ is symmetric, we can rewrite $G_0^X$ as $$G_0^X=\sum_i :\tilde x^i\ov x_i: _{(\tfrac{1}{2})}+\tfrac{1}{3} \sum_i :\theta(x^i)\ov x_i:_{(\tfrac{1}{2})}-\tfrac{1}{4k}\sum_i :\tilde x^i\ov{C_\g(x_i)} :_{(\tfrac{1}{2})}.
$$
With easy calculations  one proves that 
\begin{equation}\label{Kquadratico}
\sum_{i}:\tilde x^i  \overline 
x_i:_{(\half)}(v_\Lambda\otimes1)=v_\Lambda\otimes 
(\ov h_{\ov\Lambda})_{(-\frac{1}{2})}\cdot 1.
\end{equation}
Next we observe that, since $C_\g(x)\in \h^+$ when $x\in\g^{\ov 0}$, we have that 
\begin{equation}\label{centralpart}\sum_i :\tilde x^i\ov{C_\g(x_i)} :_{(\tfrac{1}{2})}(v_\L\otimes 1)=0.
\end{equation}
It remains to check the action of $\sum_i:\theta(x^i)\ov x_i:_{(\tfrac{1}{2})}$ on $1$. Choose the basis $\{x_i\}$ of $\g$, so that $x_i\in\g^{\ov s_i}$. We can clearly assume $s_i\in[0,1)$.
We apply formula (3.4) of \cite{KMP} to get
\begin{align*}
\sum_i:\theta(x^i)\ov x_i:_{(\tfrac{1}{2})}&=\sum_{i,m<-s_i}\theta(x^i)_{(m)}(\ov x_i)_{(-m-\tfrac{1}{2})}+\sum_{i,m\ge-s_i}(\ov x_i)_{(-m-\tfrac{1}{2})}\theta(x^i)_{(m)}\\
&-\sum_i\binom{-s_i}{1}\ov{[x^i,x_i]}_{(-\tfrac{1}{2})}.
\end{align*}
If $m<-s_i$ then  $-m>0$ so $(\ov x_i)_{(-m-\tfrac{1}{2})}\cdot 1=0$. Since $s_i\in[0,1)$, if $m>-s_i$ then $m>0$. By Lemma \ref{thetadot1}, $\theta(x^i)_{(m)}\cdot1=0$. Thus
\begin{align*}
\sum_i:\theta(x^i)\ov x_i:_{(\tfrac{1}{2})}&=\sum_{i}(\ov x_i)_{(s_i-\tfrac{1}{2})}\theta(x^i)_{(-s_i)}+\sum_is_i\ov{[x^i,x_i]}_{(-\tfrac{1}{2})}\\
&=\sum_{i}\theta(x^i)_{(-s_i)}(\ov x_i)_{(s_i-\tfrac{1}{2})}+\sum_ip(x_i)\ov{[x_i,x^i]}_{(-\tfrac{1}{2})}+\sum_is_i\ov{[x^i,x_i]}_{(-\tfrac{1}{2})}.
\end{align*}
If $s_i>0$ then $(\ov x_i)_{(s_i-\tfrac{1}{2})}\cdot 1=0$. Observe also that $\sum_ip(x_i)\ov{[x_i,x^i]}_{(-\tfrac{1}{2})}=0$. Thus
\begin{align*}
\sum_i:\theta(x^i)&\ov x_i:_{(\tfrac{1}{2})}=\sum_{i:s_i=0}\theta(x^i)_{(0)}(\ov x_i)_{(-\tfrac{1}{2})}+\sum_is_i\ov{[x^i,x_i]}_{(-\tfrac{1}{2})}.
\end{align*}
Now, applying formula (3.4) of \cite{KMP}, we get
\begin{align*}
&\theta(x^i)_{(0)}=\\&\tfrac{1}{2}\sum_{j,m<s_j-\tfrac{1}{2}}(\ov{[ x^i,x_j]})_{(m)}(\ov x^j)_{(-m-1)}+\tfrac{1}{2}p(\ov{[x^i,x_j]},\ov x^j)\sum_{j,m\ge s_j-\tfrac{1}{2}}(\ov x^j)_{(-m-1)}(\ov{[x^i,x_j]})_{(m)}\\
&-\tfrac{1}{2}\sum_j\binom{s_j-\tfrac{1}{2}}{1}(x^j,[x^i,x_j])I_{CW(\g)}.
\end{align*}
hence
\begin{align*}
\theta(x^i)_{(0)}(\ov x_i)_{(-\tfrac{1}{2})}\cdot1&=\tfrac{1}{2}\sum_{j,m<s_j-\tfrac{1}{2}}(\ov{[ x^i,x_j]})_{(m)}(\ov x^j)_{(-m-1)}(\ov x_i)_{(-\tfrac{1}{2})}\cdot1\\&+\tfrac{1}{2}p(\ov{[x^i,x_j]},\ov x^j)\sum_{j,m\ge s_j-\tfrac{1}{2}}(\ov x^j)_{(-m-1)}(\ov{[x^i,x_j]})_{(m)}(\ov x_i)_{(-\tfrac{1}{2})}\cdot1\\
&-\tfrac{1}{2}\sum_j\binom{s_j-\tfrac{1}{2}}{1}(x^j,[x^i,x_j])(\ov x_i)_{(-\tfrac{1}{2})}\cdot1.
\end{align*}
If $m< s_j-\tfrac{1}{2}$, then $-m-1\ge -s_j+\tfrac{1}{2}>-\tfrac{1}{2}$. It follows that $(\ov x^j)_{(-m-1)}(\ov x_i)_{(-\tfrac{1}{2})}\cdot1=p(\ov x_j, \ov x_i)(\ov x_i)_{(-\tfrac{1}{2})}(\ov x^j)_{(-m-1)}\cdot 1=0$. If $s_j>0$ and $m\ge s_j-\tfrac{1}{2}$ or $s_j=0$ and $m> s_j-\tfrac{1}{2}$, then $(\ov{[x^i,x_j]})_{(m)}(\ov x_i)_{(-\tfrac{1}{2})}\cdot1=p(\ov{[x^i,x_j]},\ov x_i)(\ov x_i)_{(-\tfrac{1}{2})}(\ov{[x^i,x_j]})_{(m)}\cdot1=0$.
Thus
\begin{align*}
\sum_{i:s_i=0}\theta(x^i)_{(0)}(\ov x_i)_{(-\tfrac{1}{2})}\cdot1&=\tfrac{1}{2}\sum_{i,j:s_i=s_j=0}p(\ov{[x^i,x_j]},\ov x^j)(\ov x^j)_{(-\tfrac{1}{2})}(\ov{[x^i,x_j]})_{(-\tfrac{1}{2})}(\ov x_i)_{(-\tfrac{1}{2})}\cdot1\\
&-\tfrac{1}{2}\sum_{i:s_i=0,j}\binom{s_j-\tfrac{1}{2}}{1}(x^j,[x^i,x_j])(\ov x_i)_{(-\tfrac{1}{2})}\cdot1.
\end{align*}
Next we compute
\begin{align*}
&\sum_{i:s_i=0,j}\binom{s_j-\tfrac{1}{2}}{1}(x^j,[x^i,x_j])(\ov x_i)_{(-\tfrac{1}{2})}\cdot1=\\
&-\sum_{i:s_i=0,j}\binom{s_j-\tfrac{1}{2}}{1}p(x^j,x_i)([x^j,x_j],x^i)(\ov x_i)_{(-\tfrac{1}{2})}\cdot1=\\
&-\sum_{i,j}\binom{s_j-\tfrac{1}{2}}{1}([x^j,x_j],x^i)(\ov x_i)_{(-\tfrac{1}{2})}\cdot1=\\
&-\sum_{j}\binom{s_j-\tfrac{1}{2}}{1}\ov{[x^j,x_j]}_{(-\tfrac{1}{2})}\cdot1\\&=-\sum_{j}s_j\ov{[x^j,x_j]}_{(-\tfrac{1}{2})}
.
\end{align*}
The final outcome is that
\begin{align*}
\sum_i:\theta(x^i)\ov x_i:_{(\tfrac{1}{2})}\cdot1&=\tfrac{1}{2}\sum_{i,j:s_i=s_j=0}p(\ov{[x^i,x_j]},\ov x^j)(\ov x^j)_{(-\tfrac{1}{2})}(\ov{[x^i,x_j]})_{(-\tfrac{1}{2})}(\ov x_i)_{(-\tfrac{1}{2})}\cdot1\\&+\tfrac{3}{2}\sum_is_i\ov{[x^i,x_i]}_{(-\tfrac{1}{2})}\cdot1.
\end{align*}
By \eqref{sumsixixi}, we see that 
$$
\sum_{i: 0\le s_i<1}s_i[x^i,x_i]=2\sum_{0<j\le\tfrac{1}{2}}h_{\rho^{\ov j}}.
$$
Combining this observation with  Lemma \ref{kostant}, we see that
$$
\tfrac{1}{3}\sum_i:\theta(x^i)\ov x_i:_{(\tfrac{1}{2})}\cdot1=(\ov h_{\rho_\s})_{(-\tfrac{1}{2})}\cdot 1.
$$
This, together with \eqref{Kquadratico} and \eqref{centralpart}, gives the statement.
\end{proof}

\begin{theorem}\label{sf} Let $\g$ be a basic type Lie superalgebra and let $\s$ be an indecomposable elliptic automorphism preserving the bilinear form. Let $2g$ be the eigenvalue of the Casimir operator in the adjoint representation. Let $\rho_\s$ be defined by \eqref{rhosigma} and $z(\g,\s)$ by \eqref{zgs}. Set $\rho=\rho_{Id}$.
Then we have:
\vskip5pt

 (Strange formula).
\begin{equation}\label{strangeformula}
||\rho||^2=\frac{g}{12}\sdim\g.
\end{equation}

 (Very strange formula).
\begin{equation}\label{vstrangeformula}
||\rho_\s||^2=g(\frac{\sdim\g}{12}-2z(\g,\si))
\end{equation}
\end{theorem}
\begin{rem} If $\mathfrak z(\g)$ is non-zero, then it contains an 
eigenvector of the Casimir operator with zero eigenvalue, hence $g=0$, and 
the very strange formula amounts to saying that  $\rho_\s$ is isotropic.
\end{rem}
\begin{proof} Let $\{v_i\}_{i\in \ganz_+}$ be a basis of $CW(\ov \g)$ with $v_0=1$. Write $(L^{\ov\g})^{CW(\ov\g)}_{(1)}\cdot 1=\sum_i c_i v_i$. If  $M_0$ is a highest weight module with highest weight 
$\L=-\rho_\si+k\L_0$ then $L^{\ov\g}_{(1)}(v_\L\otimes 1)=\sum c_i (v_\L\otimes v_i)$ with the coefficents $c_i$ that do not depend on $k$. By
Proposition~\ref{azioneg0},  $G_0(v_{\L}\otimes 1)=0$. 
Applying \eqref{gzeronuovo} and
Lemma~\ref{xixi}, we find that
\begin{align*} 
0&=(-\half\Vert\rho_\si\Vert^2+(k-g)z(\g,\si)-\tfrac{1}{4k}\rho_\s(C_\g(h_{\rho_\s}))-\tfrac{1}{16}(k-\frac{2g}{3})\sdim\g)(v_\L\otimes
1)\\
&-\sum_ic_ik(v_\L\otimes v_i).
\end{align*}
Since this equality holds for any $k$, we see that $c_i=0$ if $i>0$. Moreover the coefficient of $v_\L\otimes 1$ must vanish. This coefficient is
{\small
$$
\tfrac{1}{k}\left(-\tfrac{1}{4}\rho_\s(C_\g(h_{\rho_\s}))+k(-\half\Vert\rho_\si\Vert^2-gz(\g,\si)+\tfrac{g}{24}\sdim\g)+k^2(z(\g,\si)-\tfrac{1}{16}\sdim\g-c_0)\right),
$$}
so, again by the genericity of $k$, we obtain
\begin{align}
\rho_\s(C_\g(h_{\rho_\s}))&=0,\label{cgzero}\\
-\half\Vert\rho_\si\Vert^2-gz(\g,\si)+\tfrac{g}{24}\sdim\g&=0,\label{verystrangeformula}\\
z(\g,\si)-\tfrac{1}{16}\sdim\g&=c_0.\label{actionLbar}
\end{align}
Formula \eqref{verystrangeformula} is \eqref{vstrangeformula}  which specializes clearly to \eqref{strangeformula} when $\s=I_\g$.
\end{proof}
As byproduct of the proof of Theorem \ref{sf} we also obtain
\begin{prop}\label{azioneDeG} \ 
\begin{enumerate} 
\item $(L^{\ov\g})^{CW(\ov\g)}_{0}\cdot 1= 
z(\g,\si)-\tfrac{1}{16}\sdim\g$.
\item If $M$ is a highest weight $L(\g)'$-module  with highest 
weight $\L$, then
\begin{align}\notag &(G^X_0)^2(v_\L\otimes
1)=\\\label{gquadro}&\tfrac{1}{2}\left((\ov\L+2\rho_\si,\ov\L)-\tfrac{1}{2k}\ov \L(C_\g(h_{\ov \L}))+\frac{g}{12}\sdim\g-2gz(\g,\si)\right)(v_\L\otimes 
1).\end{align}
\end{enumerate}
\end{prop}
\begin{proof}We saw in the proof of Proposition \ref{sf} that $(L^{\ov\g})^{CW(\ov\g)}_{0}\cdot 1= 
c_0$ and \eqref{actionLbar} gives our formula for $c_0$. 

Again by \eqref{gzeronuovo} and Lemma~\ref{xixi}, 
\begin{align*} G_0^2(v_\L\otimes 1)&=(\half(\ov
\L+2\rho_\si,\ov\L)-\tfrac{1}{4k}\ov\L(C_\g(\ov \L))+(k-g)z(\g,\si))v_\L\otimes 1\\
&-\tfrac{1}{16}(k-\frac{2g}{3})\sdim\g(v_\L\otimes 
1)v_\L\otimes 1-kv_\L\otimes
(L^{\ov\g})^{CW(\ov\g)}_{0}\cdot 1).
\end{align*} Using the first equality we get the second claim.
\end{proof}

\vskip15pt
\footnotesize{

\noindent{\sl Victor Kac}: Department of Mathematics, Rm 2-178, MIT, 77 
Mass. Ave, Cambridge, MA 02139;\\
{\tt kac@math.mit.edu}

\noindent{\sl Pierluigi M\"oseneder Frajria}: Politecnico di Milano, Polo regionale di Como, 
Via Valleggio 11, 22100 Como,
ITALY;\\ {\tt pierluigi.moseneder@polimi.it}

\noindent{\sl Paolo Papi}: Dipartimento di Matematica, Universit\`a di Roma 
``La Sapienza", P.le A. Moro 2,
00185, Roma , ITALY;\\ {\tt papi@mat.uniroma1.it} }


\vskip15pt
\begin{thebibliography}{99.}
\bibitem{Be} S. Benayadi, \emph{Quadratic Lie superalgebras with the completely reducible action of the even part on the odd par}t, J. Algebra 223 (2000) 344--366.
\bibitem{B} J.~Burns,\emph{
An elementary proof of the "strange formula'' of Freudenthal and de Vries.}
Q. J. Math. 51 (2000), no. 3, 295--297. 
  \bibitem{KacD}
A.~De~Sole and V.~G. Kac, \emph{Finite vs affine {$W-$}algebras}. Jpn.
  J.  Math. \textbf{1} (2006), 137--261.
   \bibitem{chuah} M. Chuah, \emph{Finite order automorphisms on contragredient Lie superalgebras}, J. Algebra, \textbf{351}, (2012), 138--159.
    \bibitem{FS}  H. D.~Fegan, B.~Steer, \emph{On the ``strange formula''' of Freudenthal and de Vries.} Math. Proc. Cambridge Philos. Soc. 105 (1989), no. 2, 249--252.
  \bibitem{F}  H.~Freudenthal, H.~de Vries, \emph{Linear Lie Groups}, Academic Press, New York, 1969.
\bibitem{HP2}
J.-S. Huang and P.~Pand{\v{z}}i{\'c}, \emph{Dirac cohomology for Lie superalgebras}. Transform.
Groups \textbf{10} (2005), 201--209.

\bibitem{Kaceta1}V.~G. Kac, \emph{Infinite-dimensional Lie algebras, and the Dedekind $\eta$-function.} (Russian)
Funkcional. Anal. i Prilozen. 8 (1974), no. 1, 77--78.
\bibitem{Kacsuper}
V.~G. Kac, \emph{Lie superalgebras},
Advances in Math. \textbf{26} (1977), no. 1, 8--96. 
\bibitem{Kaceta} V.~G. Kac, \emph{
Infinite-dimensional algebras, Dedekind's $\eta$-function, classical M\"obius function and the very strange formula.}
Adv. in Math. 30 (1978), no. 2, 85--136. 
\bibitem{Kac}
V.~G. Kac, \emph{Infinite dimensional {L}ie algebras}, third ed., Cambridge
  University Press, Cambridge, 1990.

\bibitem{KMP}
V.~G. Kac, P.~M\"oseneder Frajria and P.~Papi, \emph{Multiplets of representations, twisted  Dirac operators and 
Vogan's conjecture in affine setting}, Advances in Math  \textbf{217} (2008), no. 6,  
2485--2562.
\bibitem{KW} V.~G. Kac, M.~Wakimoto, \emph{Integrable highest weight modules over affine superalgebras and number theory.}  Lie theory and geometry,  415--456, Progr. Math., 123, Birkh\"auser Boston
\bibitem{KP} V.~G. Kac, D.~Peterson \emph{
Infinite-dimensional Lie algebras, theta functions and modular forms.}
Adv. in Math. 53 (1984), no. 2, 125--264. 
\bibitem{KT}
V.~G. Kac and I.~T. Todorov, \emph{Superconformal current algebras and their
  unitary representations}, Comm. Math. Phys. \textbf{102} (1985), no.~2,
  337--347.
  \bibitem{P} P.~Papi, \emph{Dirac operators in the affine setting},
Oberwolfach Reports, Vol. 6 n.1, 2009, 835--837
\bibitem{S} V.~Serganova \emph{
Automorphisms of simple Lie superalgebras.}
Izv. Akad. Nauk SSSR Ser. Mat. 48 (1984), no. 3, 585--598.

\end{thebibliography}
\end{document}